\documentclass{amsart}
\usepackage{graphicx,amsthm,amsmath,amsfonts,amssymb,epsfig}

\newtheorem{theorem}{Theorem}[section]
\theoremstyle{plain}

\newtheorem{corollary}{Corollary}[section]

\newtheorem{definition}{Definition}[section]

\newtheorem{notation}{Notation}

\newtheorem{proposition}{Proposition}[section]
\newtheorem{remark}{Remark}[section]

\newtheorem{summary}{Summary}
\numberwithin{equation}{section}

\begin{document}
\def\dbar{\leavevmode\hbox to 0pt{\hskip.2ex \accent"16\hss}d}

\title[Geometric conemanifold structures on $\mathbb{T}_{p/q}$]{Geometric conemanifold structures on $\mathbb{T}_{p/q}$,
the result of $p/q$ surgery in the left-handed trefoil knot $\mathbb{T}$}
\author[m.t.Lozano]{Mar\'{\i}a Teresa Lozano}
\address[M.T. Lozano]{IUMA, Universidad de Zaragoza \\
Zaragoza, 50009, Spain}
\email[M.T. Lozano]{tlozano@unizar.es}
\thanks{Partially supported by grant MTM2010-21740-C02-02.}
\author[J.M. Montesinos-Amilibia]{Jos\'{e} Mar\'{\i}a Montesinos-Amilibia}
\address[J.M. Montesinos]{Dto. Geometr\'{\i}a y Topolog\'{\i}a \\
Universidad Complutense, Madrid 28080 Spain}
\email[J.M. Montesinos]{jose{\_}montesinos@mat.ucm.es}
\thanks{Partially supported by grant MTM2012-30719.}
\date{\today }
%\dedicatory{Dedicated to }

\begin{abstract}
As an example of the transitions between some of the eight geometries of
Thurston, investigated in \cite{LM2014}, we study the geometries supported by the
cone-manifolds obtained by surgery on the trefoil knot with singular set the
core of the surgery. The geometric structures are explicitly constructed.
The most interesting phenomenon is the transition from $SL(2,\mathbb{R})$-geometry
to $S^{3}$-geometry through Nil-geometry. A plot of the different
geometries is given, in the spirit of the analogous plot of Thurston for the
geometries supported by surgeries on the figure-eight knot.
\end{abstract}

\maketitle
\tableofcontents

\section{Introduction}
As an example of the transitions between some of the eight geometries of
Thurston, investigated in \cite{LM2014}, we consider those supported by the
cone-manifolds obtained by surgery on the trefoil knot with singular set the
core of the surgery. These are Seifert manifold and the singularity is a
fiber. We remark that analogous constructions can be performed with any
torus-knot or link.

To perform the construction of these geometries we proceed by steps. The
first step is to construct the holonomy maps defined in the base of the
Seifert fibration. This base is the $2$-orbifold $S_{236}$. This symbol
denotes the $2$-sphere with three singular points of isotropies of orders $2$%
, $3$ and $6$, respectively. This $2$-sphere admits geometric structures
with three singular cone-points with cone-angles of $2\pi /2$, $2\pi /3$ and
$\alpha $. For the value $\alpha =2\pi /6$ the geometry is Euclidean; for
values less than $2\pi /6$, it is hyperbolic, and for values bigger than $%
2\pi /6$, to some certain limit, it is spherical. We denote by $S_{23\alpha }
$ these geometric cone manifolds. The \textit{holonomy maps} of these
geometries are homomorphisms from the fundamental group of the triple
punctured $2$-sphere into the group of isometries of hyperbolic, Euclidean
or spherical plane, as the case may be. We reserve the term \textit{holonomy}
for the image of this homomorphism.  We give a careful description of this
holonomy map. The important fact is that we define a continuous family of
model geometric spaces, say $G_{\alpha }$ that vary with the angle $\alpha $%
. The model $G_{\alpha }$ is the Poincar\'{e} disk model of hyperbolic $2$%
-space in $\mathbb{C}\cup \{\infty \}$, with the radius of the disk varying
from $1$ to infinity. At infinity, $G_{\alpha }=\mathbb{C}$ represents the
Euclidean space, and for $\alpha $ bigger than $2\pi /6$, $G_{\alpha }=$ $%
\mathbb{C}\cup \{\infty \}$ represents the spherical geometry.

The second step in our construction is to lift these holonomy maps to
homomorphisms from the exterior of the trefoil knot into the groups of
isometries of Riemaniann geometric structures in $S^{3}$, Nil or $\widetilde{SL(2,\mathbb{R})}$  as the case may be. These geometries
form a continuous family of model geometric spaces that project onto the
family $G_{\alpha }$. The construction of this family is the content of
\cite{LM2014}, where we describe a $2$-parameter family $X(R,S)$  of Riemaniann geometries. In this paper we use the particular case  $X(S,S)$%
. The geometries $X(S,S)$, covering the spherical, Euclidean, hyperbolic $%
G_{\alpha }$, are the Thurston's geometries $S^{3}$, Nil, $SL(2,\mathbb{R})$, respectively. We
remark that each holonomy map downstairs lifts to an infinity of holonomy maps from
the exterior of the trefoil knot into the groups of isometries of the
corresponding $\ X(S,S)$ geometries.

These lifted holonomies correspond to actual geometric structures, modeled
in $X(S,S)$, of the complement of the trefoil knot. The completion of these
structures, when the completion is a $3$-manifold, are cone-manifold
structures in the result of Dehn surgery on the trefoil knot. The core of
the surgery being the singular set. These surgeries are almost always
Seifert manifolds. We introduce the following convenient notation for the
cone-manifold.
\begin{equation*}
S(m,n)=\left( Oo0|-1;(2,1),(3,1),(m,n)\right) \qquad m\geq 0
\end{equation*}%
is the Seifert manifold $\left( Oo0|-1;(2,1),(3,1),(\frac{m}{r},\frac{n}{r}%
)\right) $, where $r=\gcd (m,n)$. The exceptional fibre is $(\frac{m}{r},%
\frac{n}{r})$. This Seifert manifold is the result of performing some Dehn
surgery on the trefoil knot. The core of the surgery is the exceptional
fibre $(\frac{m}{r},\frac{n}{r})$, along which there is a cone-singularity
of angle $2\pi /r$.

The situation is exactly analogous to the one discovered by Thurston
\cite{Thurston1} for the figure-eight knot. Following him we also plot (see Figure \ref{fplot21}) the different Thurston's geometric structures with varying angle $\alpha $ for
the different Dehn surgeries of the trefoil knot.

The points in the plot bearing the symbol $\blacklozenge $ correspond to
points $(x,y)$ such that $x$ and $y$ are non-negative integers with $\gcd
(x,y)=1$. They represent the manifolds $S(x,y)$ with no singularity. The
points in the line $x/y$ connecting the origin $(0,0)$ with the point $(x,y)$
represent $S(x\times r,y\times r)$. This is a cone-manifold with underlying $%
3$-manifold the common Seifert manifold $S(x,y)$, but the cone angle $2\pi /r
$ varies from zero to infinity. Points in the right vertical dashed line
correspond to points with Nil geometry and it separates regions with $S^{3}
$ geometry and $\widetilde{{SL(2,\mathbb{R})}}$ geometry.
\begin{figure}[h]
\begin{center}
\epsfig{file=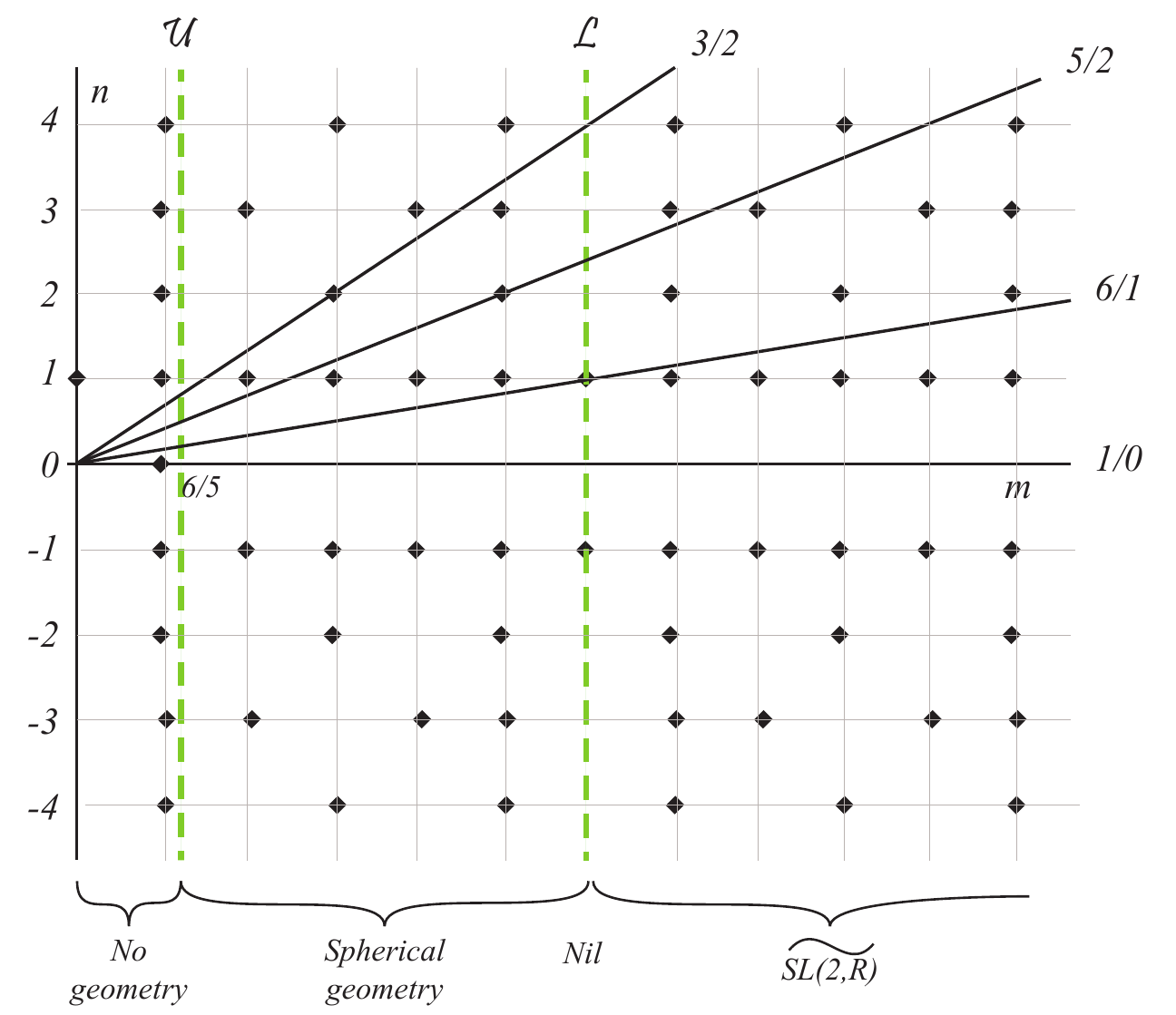,height=8cm}
\end{center}
\caption{The plot $\mathcal{P}_{2}$.}
\label{fplot21}
\end{figure}

The vertical axis $x=0$ yields $S(0,y)$, corresponding to the manifold $%
L(2,1)\sharp L(3,1)$. Between this axis (included) and the vertical line $%
m=6/5$ we find the Seifert manifolds
\begin{equation*}
S(1,n)=(Oo0|-1;(2,1),(3,1),(1,n))=(Oo0|n-1;(2,1),(3,1))
\end{equation*}
which are the lens $-L(6n-1,2n-1)$, for $n\neq 0$, and $\mathbb{S}^{3}$ for $%
n=0$. Certainly, these Seifert manifolds (lens spaces) support spherical
geometry. However, the fibre $(1,n)$ of the Seifert fibration $%
(Oo0|-1;(2,1),(3,1),(1,n))$ is not a geodesic of that geometry. This fact is
easy to understand in $S(1,0)=\mathbb{S}^{3}$, where the fibre $(1,0)$,
being a regular fibre, is the left trefoil knot, which clearly is not a
geodesic in $\mathbb{S}^{3}$. We are studying geometric structures in the
Seifert manifold $S(x,y)$ such that the fibre $(x,y)$ is geodesic (singular
or not).

For instance, the line of slope $6/1$ is the Seifert manifold
\begin{equation*}
S(m,n)=\left( Oo0|-1;(2,1),(3,1),(6,1)\right) .
\end{equation*}
The plot shows that it has Nil geometry (a well known fact) and that if
the angle in the exceptional fibre\ $(6,1)$ is less or bigger than $2\pi $
it has $SL(2,\mathbb{R})$ or $S^{3}$-geometry as the case may be. The upper
limit for the angle is $10\pi $, where the spherical geometry collapses.

We thanks Professor Porti for pointing us, after this paper was written, that \cite{K1984} contains a former approach to the geometric structures on $\mathbb{T}_{p/q}$, where the Lorentz metric (pseudo-Riemannian) on $\widetilde{SL(2,\mathbb{R})}$ is considered in order to obtain results on the corresponding volume.

\section{Conemanifold}

In this section the concept of topological and geometric 2-conemanifold will
be defined. There exist analogous concepts in other dimensions, in
particular we shall use also topological and geometric 3-conemanifold
without any new detailed definition, since they are just straightforward
generalization of the 2 dimensional case.

\subsection{Topological 2-conemanifolds}

\begin{definition}
A \emph{2-conemanifold} is a set $(\Sigma ,P,v)$, where $\Sigma $ is a
closed (compact and without boundary) surface, $P$ is a finite number of
\emph{singular} points $P=\{x_{1},...x_{k}\}\subset \Sigma $ and $v$ is a
\emph{valuation}
\begin{equation*}
v:\Sigma \longrightarrow \mathbb{R}\cup \{\infty \}
\end{equation*}%
such that $v(x)=1$ for all points but for the singular points $x_{i}\in P$
where $v(x_{i})\neq 1,\,i=1,...k$. The singular points $x_{i}\in P\subset \Sigma
$ such that $v(x_{i})<\infty $ are called \emph{conic points}. The
singular points $x_{j}\in P\subset \Sigma $ such that $v(x_{j})=\infty $ are
called \emph{cusps}.
\end{definition}

A 2-conemanifold is determined, up to homeomorphism, by $\Sigma$ and the
list $\{ v(x_{i}) | x_{i}\in P\}$.

For example, $S_{2,3,r}^{2}$, $r>1$, will denote a 2-sphere $S^{2}$ with
three singular points with valuation $2,3,r$.

The \textit{Euler characteristic} $\chi ^{c}(\Sigma ,P,v)$, \textit{of a
2-conemanifold} $(\Sigma ,P,v)$ is a real number defined by
\begin{equation*}
\chi ^{c}(\Sigma ,P,v)=\chi (\Sigma )+\underset{x\in P}{\sum }(\frac{1}{v(x)}%
-1)
\end{equation*}%
where $\chi (\Sigma )$ is the Euler characteristic of the surface $\Sigma $.
The Euler characteristic $\chi ^{c}(\Sigma ,v)$ is a topological invariant
of $(\Sigma ,P,v)$.

For instance, for an orientable surface of genus $g$ with $k$ singular
points
\begin{equation*}
\chi ^{c}(O,g|r_{1},...,r_{k})= 2-2g-k+\underset{i=1}{\overset{k}{\sum }}
\frac{1}{r_{i}}
\end{equation*}
where $r_{i}=v(x_{i})$.

\subsection{ Geometric 2-conemanifolds}

Given a 2-conemanifold $(\Sigma ,P,v)$ it is often posible to define a
geometric structure in $\Sigma \setminus P$ compatible with the valuation at
the singular points, as follows.

Let $X$ be $S^{2}$, $E^{2}$ or $H^{2}$. Let $\widehat{X}$ denote $S^{2}$
(the geometric 2-sphere), $\widehat{E^{2}}$ (the one-point compactification
of Euclidian plane) or $\widehat{H^{2}}$ (the hyperbolic plane together with
its points at infinity).

\begin{definition}
The 2-conemanifold $(\Sigma ,P,v)$ has a $X$ geometry if there exists a
finite triangulation $(K,h)$ of the closed surface $\Sigma $, where $K$ is a
2-dimensional complex and $h:K\longrightarrow \Sigma $ is a homeomorphism,
such that

\begin{enumerate}
\item $h^{-1}(x_{i})$ is a vertex of $K$, for all $x_{i}\in P$.

\item For every triangle $\sigma \in K$ there exists a homeomorphism $%
h_{\sigma }$ from a geodesic triangle $t_{\sigma }$ in $\widehat{X}$:
\begin{equation*}
h_{\sigma }:t_{\sigma }\longrightarrow \sigma
\end{equation*}%
such that

\begin{enumerate}
\item if $v(x_{i})=\infty$ and $h^{-1}(x_{i})$ is a vertex of $\sigma$, then
$(h\circ h_{\sigma })^{-1}(x_{i})\in \widehat{X}\setminus X$.

\item if $v(x_{i})\in \mathbb{R}$ and $h^{-1}(x_{i})$ is vertex of exactly $%
m $ triangles $(\sigma _{1},...,\sigma _{m})$, then the sum of the angles of
$(t_{\sigma _{j}})$ at the vertex $h_{\sigma _{j}}^{-1}(h^{-1}(x_{i}))$ for $%
j=1,...,m$, is equal to $2\pi /v(x_{i})$.
\end{enumerate}

\item If two triangles $\sigma$ and $\tau$ have a common edge, $\sigma \cap
\tau =l$, the map
\begin{equation*}
h_{\tau}^{-1} \circ h_{\sigma}|_{h_{\sigma}^{-1}(l)}:
h_{\sigma}^{-1}(l)\longrightarrow h_{\tau}^{-1}(l)
\end{equation*}
is onto and it is the restriction of an isometry of $\widehat{X}$.
\end{enumerate}

These conditions define a $X$-geometric structure in $K$, such that the
homeomorphisms $h_{\sigma }$ are isometries for all $\sigma \in K$.

Then, the homeomorphism $h$ allows us to define a geometric structure in $%
\Sigma $ making $h$ an isometry. We say that the 2-conemanifold $(\Sigma
,P,v)$ has a $X$-geometry.
\end{definition}

A topological 2-conemanifold can have non isometric geometric structures.
Consider the topological cone manifold $%
S_{(4/3,4/3,4/3,4/3,4/3,4/3,4/3,4/3)}^{2}$ which is the 2-sphere with 8
conic point with valuation 4/3. Figure \ref{fcubo} shows two different
simplicial complex that are triangulations of it. The first one is isometric
to the faces of a Euclidean cube, and the second one is isometric to the
union of two regular Euclidean octogonal polygons. Observe that these two
geometric cone manifolds are not isometric by comparing the distances
between singular points.

\begin{figure}[h]
\begin{center}
\epsfig{file=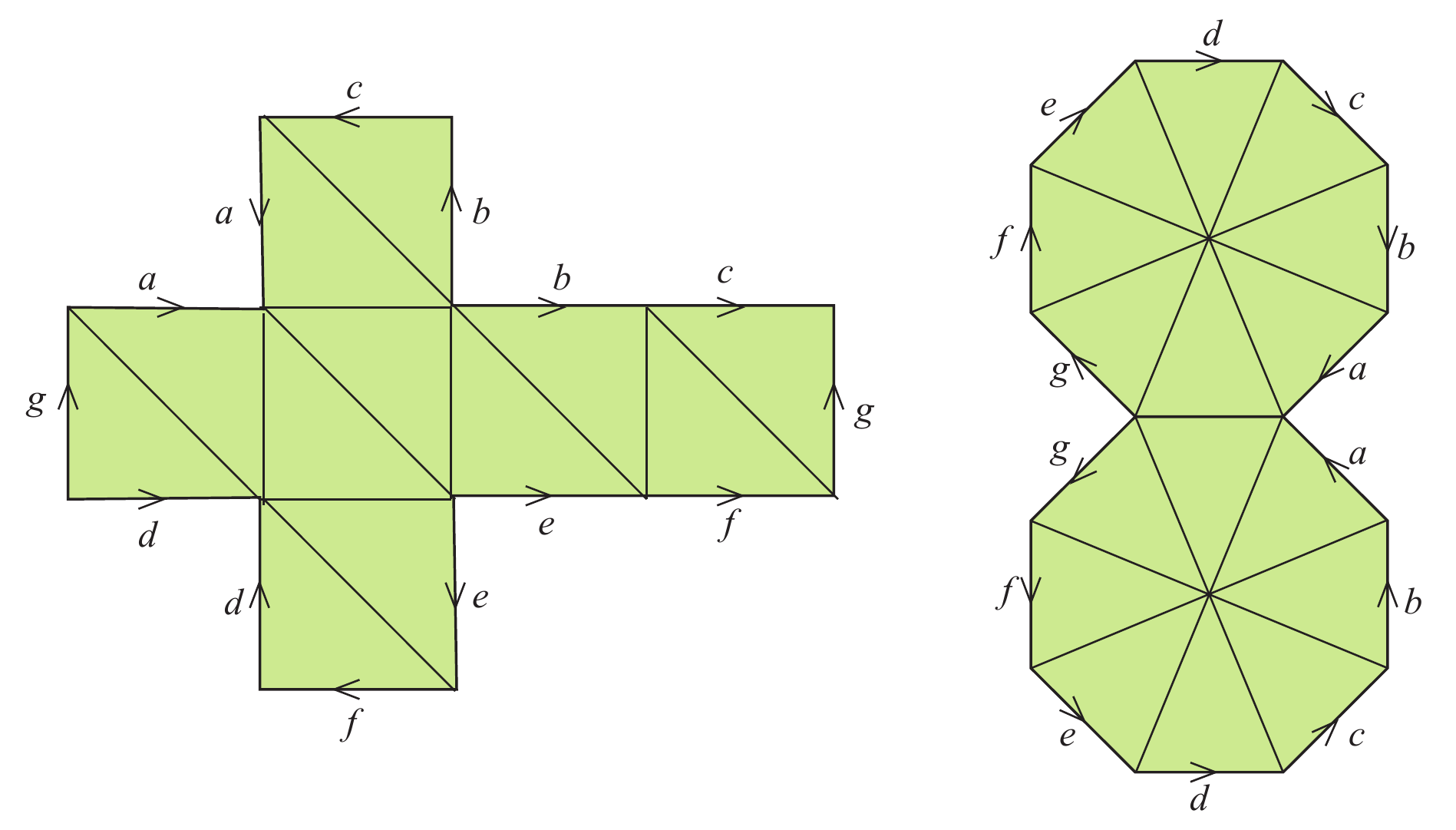,height=6cm}
\end{center}
\caption{Two geometric $S^{2}_{(4/3,4/3,4/3,4/3,4/3,4/3,4/3,4/3)}$.}
\label{fcubo}
\end{figure}

\subsection{ Values of the Euler characteristic}

The next result proves that if a 2-conemanifold has some geometric
structures, all of them are modeled on the same $X$ geometry.

Let $(\Sigma ,P,v)$ be a 2-conemanifold having a $X$-geometry. It follows
from the Gauss-Bonet that
\begin{equation*}
\int_{\Sigma }kd\sigma +\underset{i=1}{\overset{k}{\sum }}(2\pi -\frac{2\pi
}{r_{i}})=2\pi \chi (\Sigma )
\end{equation*}%
where $k$, Gauss curvature, is $-1,0,1$ if $X$ is equal to $H^{2}$, $E^{2}$
or $S^{2}$, respectively. This is equivalent to
\begin{equation*}
\int_{\Sigma }kd\sigma =2\pi (\chi (\Sigma )+\underset{i=1}{\overset{k}{\sum
}}(\frac{1}{r_{i}}-1))=2\pi \chi ^{c}(\Sigma )
\end{equation*}

\begin{proposition}
\label{xgeometry} Let $(\Sigma ,P,v)$ be a 2-conemanifold having a $X$%
-geometry. Then $\chi ^{c}(\Sigma )$ is $<0$, $=0$ or $>0$ when $X$ is equal
to $H^{2}$, $E^{2}$ or $S^{2}$ respectively. \qed
\end{proposition}

\subsection{ The 2-conemanifolds $(O,0|2,3,r)$}

Consider the 2-conemanifold $(\Sigma ,P,v)=(O,0|2,3,r)$, the 2-sphere with
three singular points and valuation $v(x_{1})=2$, $v(x_{2})=3$, and $%
v(x_{3})=r$. This notation is suggested by the Seifert notation in \cite%
{S1933}.

\begin{equation}  \label{s23r}
\chi^{c}((O,0|2,3,r))=2-3+\frac{1}{2}+\frac{1}{3}+\frac{1}{r}=-\frac{1}{6}+%
\frac{1}{r}
\end{equation}

\begin{proposition}
The 2-conemanifold $(O,0|2,3,r)$ has Euclidean geometry for $r=6$; spherical
geometry for $\frac{6}{5}<r<6$; and hyperbolic geometry for $r>6$.
\end{proposition}

\begin{proof}
By (\ref{s23r})
\begin{equation}
\chi ^{c}((O,0|2,3,r))\left\{
\begin{array}{c}
< \\
= \\
>%
\end{array}%
\right\} 0\Leftrightarrow r\left\{
\begin{array}{c}
> \\
= \\
<%
\end{array}%
\right\} 6
\end{equation}%
Therefore, Proposition \ref{xgeometry} indicates which type of geometry we
should look for.

The hyperbolic metric in the Poincar\'{e} disc model, the open unit disc
\begin{equation*}
D_{1}=\left\{ z=x+iy\;|\;x^{2}+y^{2}<1\right\}
\end{equation*}%
is given by
\begin{equation*}
ds^{2}=\frac{4dzd\overline{z}}{(1-z\overline{z})^{2}}.
\end{equation*}

In order to apply degeneration on geometric structures is more convenient to
work with a disc in $\mathbb{C}$ with radius $\frac{1}{\sqrt{S}}$. Then the
dilatation
\begin{equation*}
\begin{array}{cccc}
\lambda : & D_{\frac{1}{\sqrt{S}}} & \longrightarrow & D_{1} \\
& z & \to & \sqrt{S}z%
\end{array}%
\end{equation*}
is an isometry if and only if the metric on $D_{\frac{1}{\sqrt{S}}}$ is
given by
\begin{equation*}
ds^2=\frac{4Sdzd\overline{z}}{(1-Sz\overline{z})^{2}},\quad 1\geq S>0
\end{equation*}

\begin{figure}[h]
\begin{center}
\epsfig{file=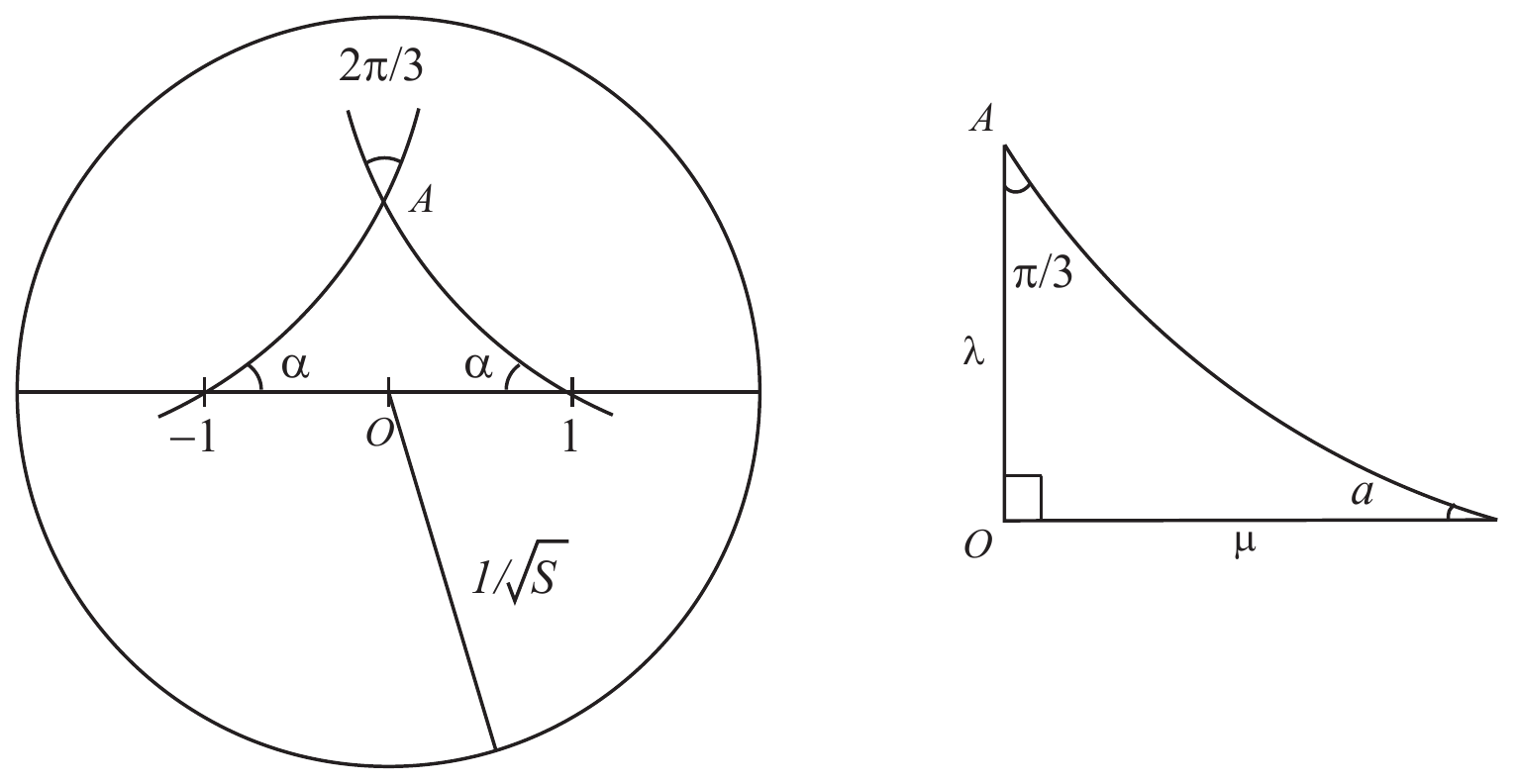,height=5cm}
\end{center}
\caption{The hyperbolic plane and a hyperbolic triangle.}
\label{fs23rhyper}
\end{figure}

Figure \ref{fs23rhyper} shows a hyperbolic triangle $T$ . The angles at the
two vertices placed at the points $(-1,0)$ and $(1,0)$ are both $\alpha $.
The remaining vertex has an angle of $2\pi /3$. Lets us relate the angle $%
\alpha $ with $S$ using well known formulas for hyperbolic triangles.
\begin{equation*}
\cosh \mu =\frac{\cos \frac{\pi }{3}}{\sin \alpha }=\frac{1}{2\sin \alpha }%
;\qquad \cosh \lambda =\frac{\cos \alpha }{\sin \frac{\pi }{3}}=\frac{2\cos
\alpha }{\sqrt{3}}
\end{equation*}

\begin{eqnarray*}
\mu &=& \int_{0}^{1}\sqrt{\frac{4S}{(1-St^{2})^{2}}}dt=\int_{0}^{1}\frac{2%
\sqrt{S}}{(1-St^{2})}dt= \\
&=& \lg (1+\sqrt{S}t)-\lg (1-\sqrt{S}t)|_{0}^{1}= \lg \frac{1+\sqrt{S}}{1-%
\sqrt{S}}
\end{eqnarray*}

\begin{equation*}
\Rightarrow \quad e^{\mu }=\frac{1+\sqrt{S}}{1-\sqrt{S}}\quad
\Rightarrow\quad \cosh \mu =\frac{e^{2\mu} +1}{2e^{\mu }}=\frac{1+S}{1-S}
\end{equation*}
Therefore
\begin{equation}  \label{esalfa}
\frac{1}{2 \sin \alpha}=\frac{1+S}{1-S}\quad \Rightarrow \boxed{S=\frac{1-2
\sin \alpha}{1+2 \sin \alpha}}
\end{equation}

Analogously, $\mathbb{C}P^{1}$ with the spherical Riemannian metric,
\begin{equation*}
ds^2=\frac{-4Sdzd\overline{z}}{(1-Sz\overline{z})^{2}},\quad S<0
\end{equation*}
is the stereographic projection of the sphere $S^{2}$ with radius $\frac{1}{%
\sqrt{-S}}$ endowed with a Riemannian metric isometric to the usual
spherical metric on the unit sphere in $\mathbb{R}^{3}$. The circle of
radius $\frac{1}{\sqrt{-S}}$ is the equator. Figure \ref{fs23rspher} shows
the spherical triangle $T$ analogous to the hyperbolic case.
\begin{figure}[h]
\begin{center}
\epsfig{file=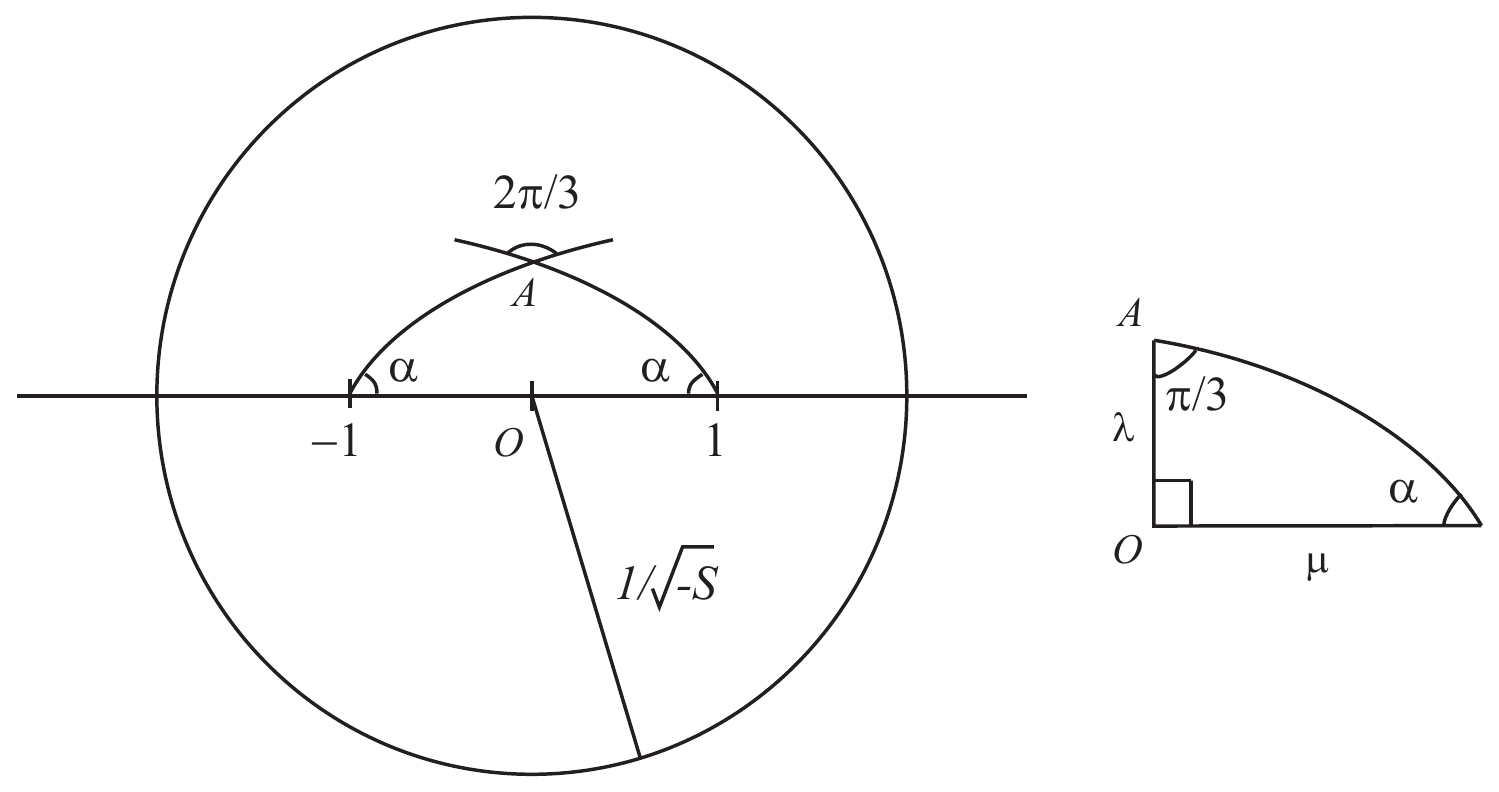,height=5cm}
\end{center}
\caption{The spherical case.}
\label{fs23rspher}
\end{figure}

In this spherical case
\begin{equation*}
\cos \mu =\frac{\cos \frac{\pi}{3}}{\sin \alpha}=\frac{1}{2\sin \alpha}
\end{equation*}

\begin{eqnarray*}
\mu = \int_{0}^{1}\sqrt{\frac{-4S}{(1-St^{2})^{2}}}dt=\int_{0}^{1}\frac{i 2%
\sqrt{S}}{(1-St^{2})}dt \\
\Rightarrow -i \mu= \lg (1+\sqrt{S}t)-\lg (1-\sqrt{S}t)|_{0}^{1}= \lg \frac{%
1+\sqrt{S}}{1-\sqrt{S}}
\end{eqnarray*}

\begin{equation*}
\Rightarrow \quad e^{-i\mu }=\frac{1+\sqrt{S}}{1-\sqrt{S}}\quad
\Rightarrow\quad \cos \mu =\frac{e^{i\mu} +e^{-i\mu}}{2}=\frac{1+S}{1-S}
\end{equation*}
Therefore as before
\begin{equation}  \label{eesealfa}
\frac{1}{2 \sin \alpha}=\frac{1+S}{1-S}\quad \Rightarrow \boxed{S=\frac{1-2
\sin \alpha}{1+2 \sin \alpha}}
\end{equation}
\begin{figure}[h]
\begin{center}
\epsfig{file=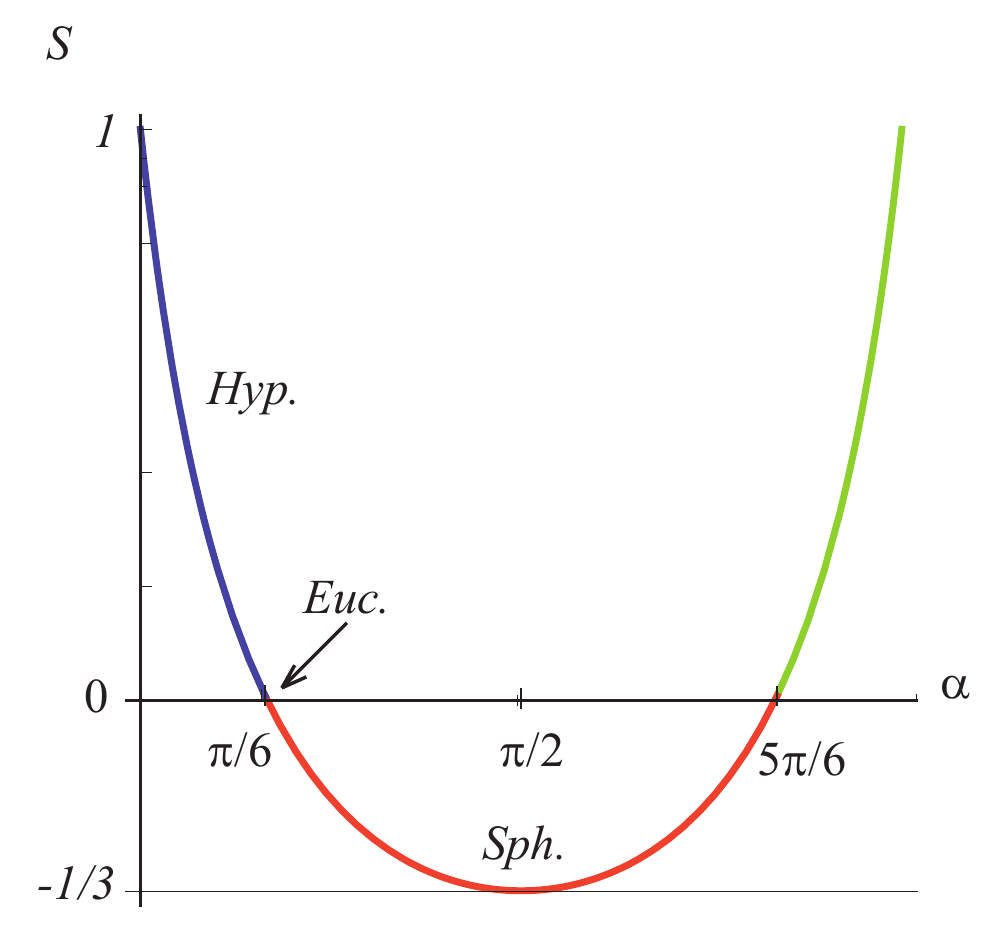,height=5cm}
\end{center}
\caption{The graphic $S(\protect\alpha )$.}
\label{fSalfa}
\end{figure}

Figure \ref{fSalfa} shows the graph of $S$ as a function of $\alpha $ and
the values where the triangle $T$ exists in hyperbolic plane (Hyp.),
Euclidian plane (Euc.) and sphere (Sph.). Therefore
\begin{equation*}
\alpha =0\quad \Leftrightarrow \quad S=1;\qquad \alpha =\frac{\pi }{6}\quad
\Leftrightarrow \quad S=0.
\end{equation*}%
\begin{figure}[h]
\begin{center}
\epsfig{file=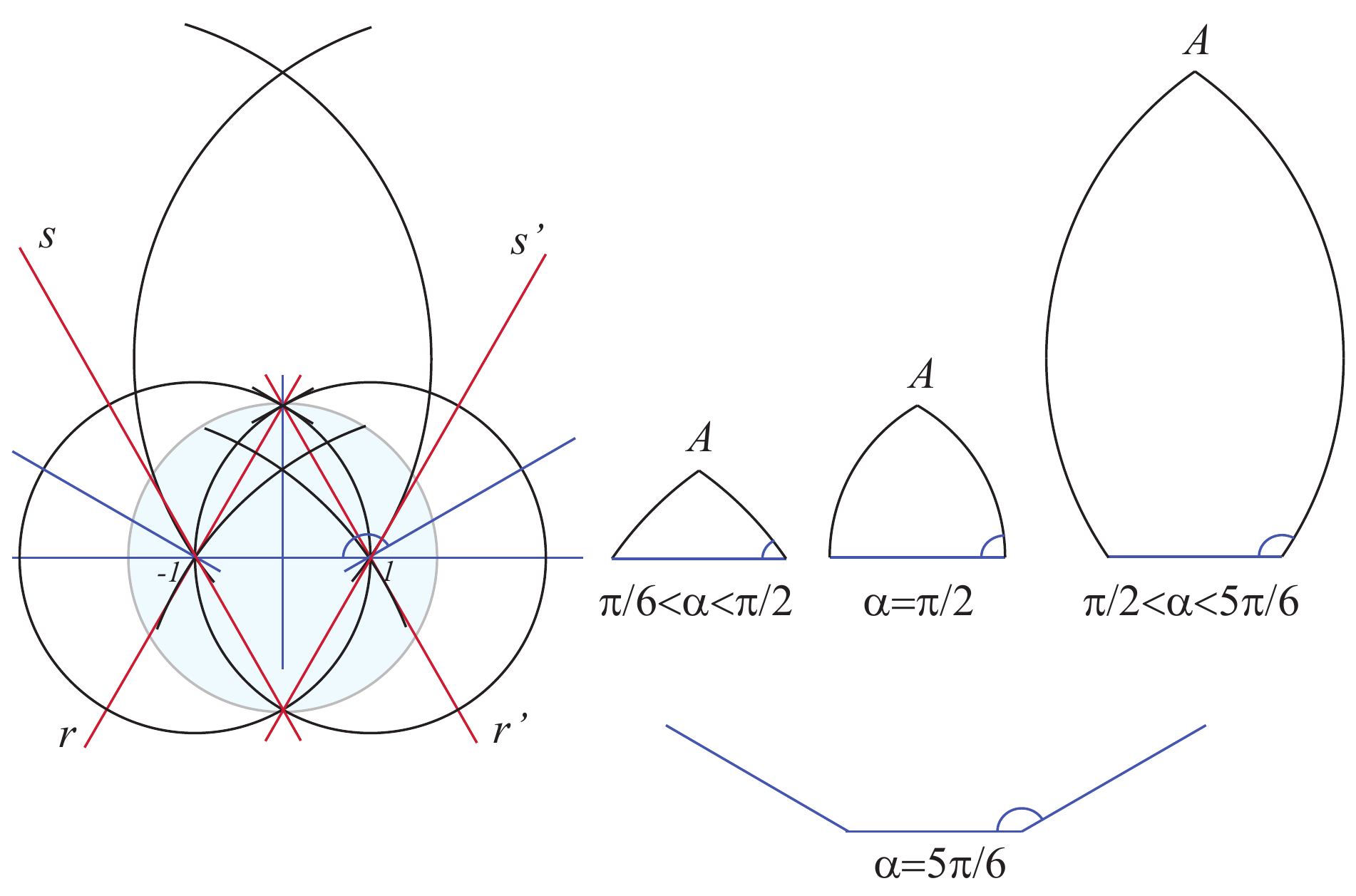,height=7cm}
\end{center}
\caption{Some spherical triangles and the limit case.}
\label{ftriangulos}
\end{figure}

The case $S=0$ correspond to the disc of infinite radius: the Euclidian
case. The hyperbolic triangle $T$ exists for $0<\alpha <\frac{\pi}{6}$; for $%
\alpha =\frac{\pi}{6}$ is an Euclidean triangle and for $\frac{\pi}{6}%
<\alpha <\frac{5\pi}{6}$ is a spherical triangle.

Figure \ref{ftriangulos} shows some cases in the stereographic projection of
the sphere $S^{2}$ of radius $\frac{1}{\sqrt{-S}}$ in the plane, where the
shadowed disc corresponds to one of the hemispheres of $S^{2}$ limited by
the equator (in the figure $S=-\frac{1}{3}$). The edge between the two
angles $\alpha $ is common to all the triangles, and it is the segment lying
between points $(-1,0)$ and $(1,0)$. The other two edges are part of circles
centered at point in the lines $r$ and $r^{\prime }$, for $\frac{\pi }{6}%
<\alpha \leq \frac{\pi }{2}$, and in the lines $s$ and $s^{\prime }$ for $%
\frac{\pi }{2}\leq \alpha \leq \frac{5\pi }{6}$. The radius of this circles
is a function of $S$. In the same figure are depicted three triangles, and
also the limit case $S=0,\,\alpha =\frac{5\pi }{6}$ where the geometry
becomes Euclidean.

The geometric 2-conemanifold $(O,0|2,3,r)$, where $r=\frac{\pi }{\alpha }$,
is obtained as the quotient of $T$ by identifying the two edges, meeting in $%
A$, by a $(2\pi /3)$-rotation isometry centered at $A$ and identifying the
two halfs of the other edge by a $\pi $-rotation isometry centered at $O$.
\end{proof}

\subsection{ The holonomy}

Let $(\Sigma ,P,v)$ be a 2-conemanifold having an $X$-geometry. The \textit{%
holonomy map of }$(\Sigma ,P,v)$  is a homomorphism
\begin{equation*}
\omega :\pi _{1}(\Sigma \setminus P)\longrightarrow I\!so(X)
\end{equation*}%
defined as follows.

Let $\gamma $ be a meridian of $x\in P$. Develop $\Sigma $ over $X$ along $%
\gamma $; the isometry relating the two ends of this developing is, by
definition, $\omega (\gamma )$.

The image of $\omega $ is called the \textit{holonomy} of $(\Sigma ,P,v)$.

Thurston (\cite{Thurston}, \cite{MatsuMonte1991}) proved that if the valuations
are all natural numbers, the conemanifold is an orbifold obtained as the
quotient of $X$ by the  holonomy.

\begin{proposition}
\label{paybend} Consider the 2-conemanifold $(O,0|2,3,r)$ with its
corresponding $X$-geometry. Let $\rho $ denote the rotation of angle $\frac{%
2\pi }{r}$around the point $0$. Let $\tau _{\pm }$ denote the translation
sending $0$ to $\pm 1$ in the model of radius $\frac{1}{\sqrt{|S|}}$. This
translation will be hyperbolic, Euclidean or spherical, according to $X$.
Then the image of the holonomy of the $X$-geometry of $(O,0|2,3,r)$ is the
subgroup of $Iso^{+}X$ generated by the rotations $a$ and $b$, conjugate to $%
\rho $ by $\tau _{\pm }$.
\end{proposition}

\begin{proof}
The fundamental group $\pi _{1}(\Sigma \setminus P)$ is the fundamental
group of a sphere with three punctures.
\begin{equation*}
\pi _{1}(\Sigma \setminus P)=|c,d:-|
\end{equation*}%
where $c,d$ are meridians of the conic points with valuation $3$ and $2$
respectively.

The holonomy map is, by definition, a homomorphism
\begin{equation}
\begin{array}{cccc}
\omega : & \pi _{1}(\Sigma \setminus P) & \longrightarrow & I\!so(X) \\
& c & \rightarrow & \omega (c) \\
& d & \rightarrow & \omega (d)%
\end{array}
\label{eholonomia}
\end{equation}%
where $(\omega (c))^{3}=I\!dentity$ and $(\omega (d))^{2}=I\!dentity$.
Therefore the holonomy map $\omega $ factors through the group
\begin{equation*}
|c,d:c^{3}=d^{2}=1|
\end{equation*}%
We change the above presentation using $a=dc^{-1}$ and $b=ccd^{-1}$. Then $%
d=ac$ and $c=ba,\quad \Rightarrow d=aba$.
\begin{equation*}
|c,d:c^{3}=d^{2}=1|=|a,b:(ba)^{3}=(aba)^{2}=1|
\end{equation*}

On the other hand $a$ and $b$ are conjugate: $d^{-1}bd=d^{-1}cc=dc^{-1}=a$.
Therefore $\omega (a)$ is the rotation of angle $\frac{2\pi }{r}$ around the
point $1\in D_{S}$, and $\omega (b)$ is its conjugate by the rotation of $%
\pi $ around the point $0\in D_{S}$, or equivalently, $\omega (b)$ is the
rotation of angle $\frac{2\pi }{r}$ around the point $-1\in D_{S}$.
\end{proof}

\section{Some 3-dimensional holonomies}

The geometry $(X_{(S,S)},Q)$, is a particular case of the geometry $(X_{(R,S)},Q)$ studied in \cite{LM2014}. We recall this particular geometry
in the following subsection. In the remain of the section we will lift the
holonomies of $(O,0|2,3,r)$, constructed in the above section, to the  group of  isometries of $(X_{(S,S)},Q)$.

\subsection{The Riemanian geometry $(X_{(S,S)},Q)$}

The matrix product on the following set of $2\times 2$ complex matrices
\begin{equation*}
X_{(S,S)}=\left\{ \left[
\begin{array}{cc}
t-iSz & \sqrt{S}(x+iy) \\
\sqrt{S}(x-iy) & t+iSz
\end{array}%
\right] \,\left\vert \,x,y,z,t\in \mathbb{R}%
,t^{2}+S^{2}z^{2}-S(x^{2}+y^{2})=1\right. \right\}
\end{equation*}%
provides a Lie group structure on the 3-dimensional quadric
\begin{equation*}
X_{(S,S)}=\left\{ (x,y,z.t)\in \mathbb{R}^{4}\,|%
\,t^{2}+S^{2}z^{2}-S(x^{2}+y^{2})=1\right\}
\end{equation*}%
contained in $\mathbb{R}^{4}$.

It is proved in \cite{LM2014} that $X_{(S,S)}$   is isomorphic to the 3-sphere
if $S<0$ and it is isomorphic to $SL(2,\mathbb{R})$ if $S>0$. The limit Lie
group when $S\rightarrow 0$
\begin{equation*}
X_{1}=\lim_{S\rightarrow 0}X_{(S,S)}
\end{equation*}%
is isomorphic to the Heisenberg group.

The metric $Q$ in $X_{(S,S)}$ is the left invariant metric defined, in the
canonical basis $e_{1}=(1,0,0,1),e_{2}=(0,1,0,1),e_{3}=(0,0,1,1)$ at the
identity $(0,0,0,1)$, by the identity matrix $<1,1,1,1>$. Observe that $(X_{-1,-1},Q)$ is the spherical geometry $S^{3}$, $(X_{1,1},Q)$ is the $\widetilde{SL(2,\mathbb{R})}$ geometry and $(X_{1},Q)$ is a Nil geometry. Therefore, in this way we
can study continuos transitions between some of the Thurston's geometries. Namely,
Spherical-Nil-$\widetilde{SL(2,\mathbb{R})}$.

Each element $q\in X_{(S,S)}$ defines an isometry $l_{q}$ by left product.
The right product $r_{q^{\prime }}$ by any diagonal element $q^{\prime }$ of
$X_{(S,S)}$ is also an isometry. The \emph{left-right notation} for an
isometry is a pair of elements of  $X_{(S,S)}$
$(q,q^{\prime })$, where $q$ acts by left product and $q^{\prime }$ acts by
right product. The composition of such isometries is given by
\begin{equation*}
(q_{1},q_{1}^{\prime })\cdot (q_{2},q_{2}^{\prime
})=(q_{2}q_{1},q_{1}^{\prime }q_{2}^{\prime })
\end{equation*}%
These isometries can be expressed also as the restriction to $X_{(S,S)}$ of
linear maps in $\mathbb{R}^{4}$ and we denote by $lm(q)$ and $rm(q^{\prime
}) $ the corresponding $4\times 4$ matrices. This second notation is
convenient when we consider the limit situation $S\rightarrow 0$.

The manifold $X_{(S,S)}$ has a Seifert fibered structure, where the $\mathbb{%
S}^{1}$-action is given by the right product
\begin{equation*}
\begin{array}{ccc}
X(S,S)\times \mathbb{S}^{1} & \longrightarrow & X(S,S) \\
\begin{bmatrix}
\ast & \sqrt{R}(x+iy) \\
\ast & t+iSz%
\end{bmatrix}%
\begin{bmatrix}
e^{-i\theta } & 0 \\
0 & e^{i\theta }%
\end{bmatrix}
& \rightarrow &
\begin{bmatrix}
\ast & e^{i\theta }\left( \sqrt{R}(x+iy)\right) \\
\ast & e^{i\theta }\left( t+iSz\right)%
\end{bmatrix}%
,%
\end{array}%
\end{equation*}

Its base space $D_{S}$ is $\mathbb{C}P^{1}$ for $S<0$, and $D_{S}=\left\{
w\in \mathbb{C}P^{1};w\overline{w}<\frac{1}{S}\right\} $ for $S>0$. The
projection of the Seifert fibered structure $X_{(S,S)}$ is given by
\begin{equation*}
\begin{array}{ccc}
p:X_{(S,S)} & \longrightarrow & D_{S} \\
(x,y,z,t) & \rightarrow & \frac{x+iy}{t+iSz}%
\end{array}%
.
\end{equation*}

The action of $l_{q}$ on $X_{(S,S)}$, where $q=\left[
\begin{array}{cc}
d-iSc & \sqrt{S}(a+ib) \\
\sqrt{S}(a-ib) & d+iSc
\end{array}
\right]$, projects onto the homography
\begin{equation*}
\begin{array}{ccc}
h_{q}:D_{S} & \longrightarrow & D_{S} \\
w & \rightarrow & \frac{(d-iSc)w+(a+bi)}{S(a-bi)w+(d+iSc)}%
\end{array}%
\end{equation*}
of $D_{S}$.

For our purposes it is convenient to write the metric matrix $Q$ of $%
X_{(S,S)}$ in Seifert product coordinates $(\mu ,\nu ,\zeta )\in D_{S}\times
\mathbb{S}^{1}$. The metric matrix of $Q$ in these coordinates is the
following:
\begin{equation}
Q((\mu ,\nu ,\theta )=\left[
\begin{array}{ccc}
\frac{\nu ^{2}+1}{\left( 1-S\left( \mu ^{2}+\nu ^{2}\right) \right) ^{2}} & -%
\frac{\mu \nu }{\left( 1-S\left( \mu ^{2}+\nu ^{2}\right) \right) ^{2}} &
\frac{\nu }{S\left( 1-S\left( \mu ^{2}+\nu ^{2}\right) \right) } \\
&  &  \\
\ast & \frac{\mu ^{2}+1}{\left( 1-S\left( \mu ^{2}+\nu ^{2}\right) \right)
^{2}} & -\frac{\mu }{S\left( 1-S\left( \mu ^{2}+\nu ^{2}\right) \right) } \\
&  &  \\
\ast & \ast & \frac{1}{S^{2}}%
\end{array}%
\right] .  \label{maqSc}
\end{equation}%
Note that this formula makes sense for all $X_{(S,S)}$ for $S>0$. For $S<0$
it applies only out of the 1-sphere $C_{\infty }=\left\{ (x,y,0,0)\in
X_{(S,S)}|-S(x^{2}+y^{2})=1\right\} $.

Observe that in these coordinates the value of $Q$ do not depend on the
third coordinate $\zeta$, it only depends on the coordinates in the base of
the Seifert fibration. It is a fibred metric.

It is an easy exercise to check that the $2\times 2$ principal submatrix of (%
\ref{maqSc}), which is the restriction to the base of the Seifert manifold $%
X_{(-1,1)}$, coincides with the hyperbolic metric on the upper sheet $UH$ of
the hyperboloid $-x_{1}^{2}-x_{2}^{2}+x_{3}^{2}=1/S$ for $S>0$, and with the
usual metric on the 2-sphere $\mathbb{S}_{1/\sqrt{|S|}}^{2}$ of radius $1/%
\sqrt{|S|}$ for $S<0$. Indeed, in case $S>0$, it is enough to consider the
pullback of the usual metric in $UH$ by the map:
\begin{equation*}
\begin{array}{cccc}
\pi _{UH}: & \mathbb{R}^{2} & \longrightarrow & UH=\left\{
(x_{1},x_{2},x_{3})\in \mathbb{R}%
^{3}|-x_{1}^{2}-x_{2}^{2}+x_{3}^{2}=1/S,x_{3}>0\right\} \\
& (\mu ,\nu ) & \rightarrow & \left( \frac{\mu }{\sqrt{1-S(\mu ^{2}+\nu ^{2})%
}},\frac{\nu }{\sqrt{1-S(\mu ^{2}+\nu ^{2})}},\frac{1}{\sqrt{S(1-S(\mu
^{2}+\nu ^{2}))}}\right)%
\end{array}%
\end{equation*}%
In case $S<0$, consider the pullback of the usual metric in the north
hemisphere $\mathbb{S}_{+}^{2}$ of $\mathbb{S}_{1/\sqrt{|S|}}^{2}$ (induced
by the Euclidean metric in $\mathbb{R}^{3}$) by the map:

\begin{equation*}
\begin{array}{cccc}
\pi _{\mathbb{S}_{+}^{2}}: & \mathbb{R}^{2} & \longrightarrow & \mathbb{S}%
_{+}^{2}=\left\{ (x_{1},x_{2},x_{3})\in \mathbb{R}%
^{3}|x_{1}^{2}+x_{2}^{2}+x_{3}^{2}=1/|S|,x_{3}>0\right\} \\
& (\mu ,\nu ) & \rightarrow & \left( \frac{\mu }{\sqrt{1-(|S|(\mu ^{2}+\nu
^{2}))}},\frac{\nu }{\sqrt{1-(|S|(\mu ^{2}+\nu ^{2}))}},\frac{1}{\sqrt{%
|S|(1-|S|(\mu ^{2}+\nu ^{2}))}}\right)%
\end{array}%
\end{equation*}

These maps are the projections from the origen $(0,0,0)$ of the tangent
plane at $(0,0,1/S)$ to $UH$ (or $\mathbb{S}_{+}^{2}$, as the case may be).
Their inverse maps define the Klein-Beltrami models of the hyperbolic and
spherical 2-dimensional geometries, respectively (see Figure \ref%
{fKleinModel}). Therefore the metric on the base of the Seifert structure of
$X_{(S,S)}$, in coordinates $(\mu ,\nu )\in D_{S}$, is the usual one
of the hyperbolic geometry for $S>0$, and of the spherical geometry for $S<0$.

\begin{figure}[h]
\label{fKleinModel}
\par
\begin{center}
\epsfig{file=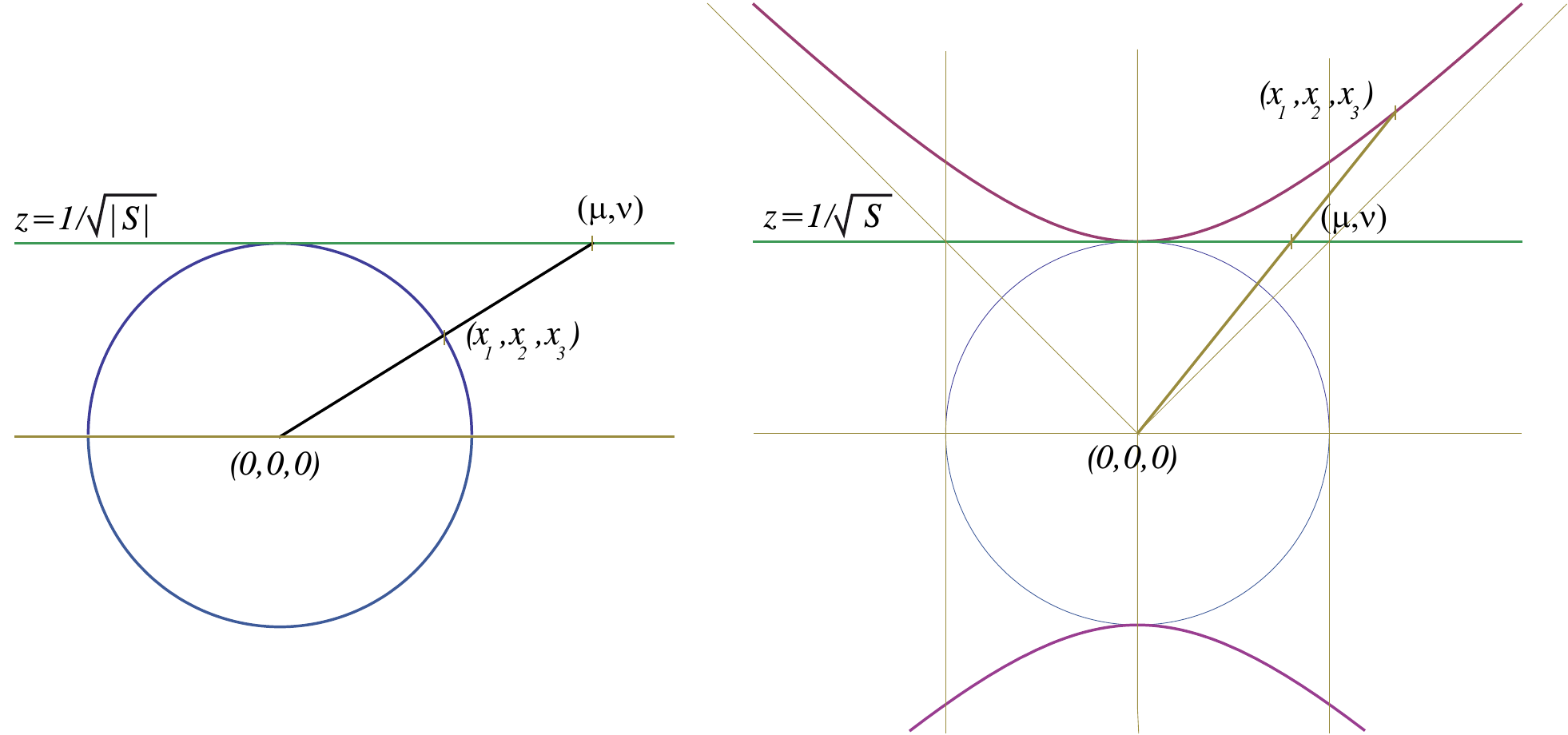,height=5.5cm}
\end{center}
\caption{The projections $\protect\pi_{NH}$ and $\protect\pi_{UH}$.}
\end{figure}

\subsection{ Lifting holonomies}

In this section we find the subgroups of isometries of $(X_{(S,S)},Q)$ that
project onto  the holonomies of $(O,0|2,3,r)$ under the
projection
\begin{equation*}
p:X_{(S,S)}\longrightarrow D_{S}.
\end{equation*}%
In a different section we will realize this lifted holonomies by
constructing suitable $3$-conemanifolds modelled on $(X_{(S,S)},Q)$.

We first define two isometries in $(X_{(S,S)},Q)$, $S\neq 0$, that projet
onto the elements $a$ and $b$, defined in Proposition \ref{paybend}, under
the map
\begin{equation*}
p:X_{(S,S)}\longrightarrow D_{S}
\end{equation*}%
as follows.

Consider the isometry $R(\alpha ,\theta )$ of $(X_{(S,S)},Q)$, given in
left-right notation by
\begin{equation}
R(\alpha ,\theta )=\left( \left[
\begin{array}{cc}
e^{i\alpha } & 0 \\
0 & e^{-i\alpha }
\end{array}%
\right] ,\left[
\begin{array}{cc}
e^{-i\theta } & 0 \\
0 & e^{i\theta }
\end{array}%
\right] \right) \quad \overset{p}{\longrightarrow }\quad w=e^{2i\alpha }z.
\label{lralfateta}
\end{equation}%
As a linear map it is given by the following matrix
\begin{equation}
lm(R(\alpha ,\theta ))=\left[
\begin{array}{cc}
\begin{array}{cc}
\cos (\alpha +\theta ) & -\sin (\alpha +\theta ) \\
\sin (\alpha +\theta ) & \cos (\alpha +\theta )%
\end{array}
& 0 \\
0 &
\begin{array}{cc}
\cos (\alpha -\theta ) & \frac{-\sin (\alpha -\theta )}{S} \\
S\sin (\alpha -\theta ) & \cos (\alpha -\theta )%
\end{array}
\end{array}%
\right]  \label{ralfateta}
\end{equation}%
The projection $p:X_{(S,S)}\longrightarrow D_{S}$ maps this isometry $%
R(\alpha ,\theta )$ onto a rotation of angle $2\alpha $ around the origin $%
O\in D_{S}$.

The translations in $X_{(S,S)}$ sending the point $(0,0,0,1)$, (such that $%
p(0,0,0,1)=0\in D_{S}$), to the point $\frac{1}{\sqrt{|1-S|}}(\pm 1,0,0,1)$
(such that $p(\frac{1}{\sqrt{|1-S|}}(\pm 1,0,0,1))=\pm 1\in D_{S}$) are
respectively
\begin{equation}
\begin{array}{llll}
t_{1} & =\left( \frac{1}{\sqrt{|1-S|}}\left[
\begin{array}{cc}
1 & \sqrt{S} \\
\sqrt{S} & 1
\end{array}%
\right] ,\left[
\begin{array}{cc}
1 & 0 \\
0 & 1
\end{array}%
\right] \right) & \overset{p}{\longrightarrow } & w=\frac{z+1}{Sz+1} \\
t_{-1} & =\left( \frac{1}{\sqrt{|1-S|}}\left[
\begin{array}{cc}
1 & -\sqrt{S} \\
-\sqrt{S} & 1
\end{array}%
\right] ,\left[
\begin{array}{cc}
1 & 0 \\
0 & 1
\end{array}%
\right] \right) & \overset{p}{\longrightarrow } & w=\frac{z-1}{-Sz+1}%
\end{array}%
.  \label{lrt1ytm1}
\end{equation}%
In linear matrix notation they are

\begin{eqnarray}  \label{et1ytm1}
lm(t_{1})&=& \frac{1}{\sqrt{|1-S|}} \left[
\begin{array}{cccc}
1 & 0 & 0 & 1 \\
0 & 1 & S & 0 \\
0 & 1 & 1 & 0 \\
S & 0 & 0 & 1
\end{array}
\right]; \\
lm(t_{-1})&=& \frac{1}{\sqrt{|1-S|}} \left[
\begin{array}{cccc}
1 & 0 & 0 & -1 \\
0 & 1 & -S & 0 \\
0 & -1 & 1 & 0 \\
-S & 0 & 0 & 1
\end{array}
\right].
\end{eqnarray}

Define $a(\alpha ,\theta )$ and $b(\alpha ,\theta )$ as the conjugate
elements of $R(\alpha ,\theta )$ by the translations $t_{1}$ and $t_{-1}$
respectively. Then the projections by $p$ of $a(\alpha ,\theta )$ and $%
b(\alpha ,\theta )$ coincide, respectively, with the elements $a$ and $b$
defined in Proposition \ref{paybend}, where $\alpha =\pi /r$.

\begin{eqnarray}  \label{ebalfateta}
a(\alpha ,\theta ) &=& t_{1}\cdot R(\alpha ,\theta )\cdot t_{1}^{-1} \\
b(\alpha ,\theta ) &=& t_{-1}\cdot R(\alpha ,\theta )\cdot t_{-1}^{-1}
\end{eqnarray}

After some computations and substitution of the value of $S$ as a function
of $\alpha $, by means of equations (\ref{eesealfa}) and (\ref{esalfa}), the
expressions for $a(\alpha ,\theta )$ and $b(\alpha ,\theta )$ in left-right
notation become
\begin{eqnarray*}
a(\alpha ,\theta ) &=&(M,R) \\
b(\alpha ,\theta ) &=&(N,R)
\end{eqnarray*}%
where
\begin{eqnarray*}
M &=&\frac{1}{2}\left[
\begin{array}{cc}
2\cos (\alpha )+i & -i\sqrt{2\cos (2\alpha )-1} \\
i\sqrt{2\cos (2\alpha )-1} & 2\cos (\alpha )-i%
\end{array}%
\right] \\
N &=&\frac{1}{2}\left[
\begin{array}{cc}
2\cos (\alpha )+i & i\sqrt{2\cos (2\alpha )-1} \\
-i\sqrt{2\cos (2\alpha )-1} & 2\cos (\alpha )-i%
\end{array}%
\right] \\
R &=&\left[
\begin{array}{cc}
e^{-i\theta } & 0 \\
0 & e^{i\theta }
\end{array}%
\right]
\end{eqnarray*}%
and as $4\times 4$ matrices they are
\begin{equation}
lm(a(\alpha ,\theta ))=\frac{1}{2}\left[
\begin{array}{cc}
A_{11} & A_{12} \\
A_{21} & A_{22}
\end{array}%
\right]  \label{eaAij}
\end{equation}%
and
\begin{equation}
lm(b(\alpha ,\theta ))=\frac{1}{2}\left[
\begin{array}{cc}
A_{11} & -A_{12} \\
-A_{21} & A_{22}
\end{array}%
\right] ,  \label{ebBij}
\end{equation}%
where
\begin{eqnarray*}
A_{11} &=&\left(
\begin{array}{cc}
2\cos (\alpha )\cos (\theta )-\sin (\theta ) & -2\cos (\alpha )\sin (\theta
)-\cos (\theta ) \\
2\cos (\alpha )\sin (\theta )+\cos (\theta ) & 2\cos (\alpha )\cos (\theta
)-\sin (\theta )
\end{array}%
\right) \\
&& \\
A_{12} &=&\left(
\begin{array}{cc}
(1-2\sin (\alpha ))\cos (\theta ) & (1+2\sin (\alpha ))\sin (\theta ) \\
(1-2\sin (\alpha )\sin (\theta ) & -(1+2\sin (\alpha ))\cos (\theta )
\end{array}%
\right) \\
&& \\
A_{21} &=&\left(
\begin{array}{cc}
(1+2\sin (\alpha ))\cos (\theta ) & -(1+2\sin (\alpha ))\sin (\theta ) \\
-(1-2\sin (\alpha ))\sin (\theta ) & -(1-2\sin (\alpha ))\cos (\theta )
\end{array}%
\right) \\
&& \\
A_{22} &=&\left(
\begin{array}{cc}
2\cos (\alpha )\cos (\theta )+\sin (\theta ) & \frac{(1+2\sin (\alpha
))(2\cos (\alpha )\sin (\theta )-\cos (\theta ))}{1-2\sin (\alpha )} \\
\frac{(1-2\sin (\alpha ))(-2\cos (\alpha )\sin (\theta )+\cos (\theta ))}{%
1+2\sin (\alpha )} & 2\cos (\alpha )\cos (\theta )+\sin (\theta )
\end{array}%
\right) .
\end{eqnarray*}

The two elements $a(\alpha ,\theta )$ and $b(\alpha ,\theta )$ are the
generators, for each $\theta $, of a group of isometries of $X_{(S,S)}$, $%
S\neq 0$, that projects onto the holonomy of the conemanifold $%
(O,0|2,3,r)$, where $\alpha =\frac{\pi }{r}$.

The case $S=0$ is a limit case. By equations (\ref{eesealfa}) and (\ref%
{esalfa})
\begin{equation*}
S=0 \Longleftrightarrow \alpha =\frac{\pi}{6}
\end{equation*}

The natural definition is
\begin{eqnarray}
lm(a(\frac{\pi}{6} ,\theta ))&=\lim _{\alpha \to \frac{\pi}{6}}lm( a(\alpha
,\theta)) \\
lm(b(\frac{\pi}{6} ,\theta ))&=\lim _{\alpha \to \frac{\pi}{6}}lm( b(\alpha
,\theta))
\end{eqnarray}

The values of all the elements in these matrices are finite but the value of
the element $(1,2)$ in $A_{22}$ for $\alpha =\pi /6$, which is
\begin{equation}
\lim_{\alpha \rightarrow \frac{\pi }{6}}\frac{-(1+2\sin (\alpha ))(-2\cos
(\alpha )\sin (\theta )+\cos (\theta ))}{1-2\sin (\alpha )}  \label{ea34} \\
=\frac{-2(-\sqrt{3}\sin (\theta )+\cos (\theta ))}{0},
\end{equation}%
becomes infinite or indeterminate according to the value of $\theta $. On
the other hand the value of the element $(2,2)$ in $A_{2,2}$ which is
\begin{equation*}
\sqrt{3}\cos (\theta )+\sin (\theta )
\end{equation*}%
should be equal to $1$ in order that $a(\frac{\pi }{6},\theta )$ and $b(%
\frac{\pi }{6},\theta )$ be isometries of the group $X_{1}$. Therefore
\begin{equation*}
\sqrt{3}\cos (\theta )+\sin (\theta )=1\Longleftrightarrow \theta =\frac{\pi
}{6}
\end{equation*}%
for this value of $\theta $ the element (\ref{ea34}) is indeterminate and
can be obtained by the l'Hopital rule.

Observe that if $\theta =\alpha +t(6\alpha -\pi)$, $t\in \mathbb{R}$, then
\begin{equation*}
\alpha \to \frac{\pi}{6}\quad \Longrightarrow \quad \theta \to \frac{\pi}{6}
\end{equation*}

Lets define
\begin{eqnarray}  \label{eatbt}
a_{t} =&\lim_{\alpha \to \frac{\pi}{6}}lm(a(\alpha,\alpha +t(6\alpha -\pi)))
= \left[
\begin{array}{cccc}
\frac{1}{2} & -\frac{\sqrt{3}}{2} & 0 & \frac{1}{2} \\
\frac{\sqrt{3}}{2} & \frac{1}{2} & 0 & -\frac{\sqrt{3}}{2} \\
\frac{\sqrt{3}}{2} & -\frac{1}{2} & 1 & -\frac{1}{2} \sqrt{3} (8 t+1) \\
0 & 0 & 0 & 1%
\end{array}
\right]  \notag \\
\\
b_{t} =& \lim_{\alpha \to \frac{\pi}{6}}lm(b(\alpha,\alpha +t(6\alpha
-\pi)))= \left[
\begin{array}{cccc}
\frac{1}{2} & -\frac{\sqrt{3}}{2} & 0 & -\frac{1}{2} \\
\frac{\sqrt{3}}{2} & \frac{1}{2} & 0 & \frac{\sqrt{3}}{2} \\
-\frac{\sqrt{3}}{2} & \frac{1}{2} & 1 & -\frac{1}{2} \sqrt{3} (8 t+1) \\
0 & 0 & 0 & 1%
\end{array}
\right]  \notag
\end{eqnarray}

For each $t\in \mathbb{R}$, $a_{t}$ and $b_{t}$ are the generators of a
subgroup of isometries in $X_{1}$ which projets on the holonomy of the
conemanifold $(O0|2,3,6)$.

\subsection{Representations of the trefoil knot group}

\begin{proposition}
Let $\mathbb{T}$ be the oriented left-handed trefoil knot in $\mathbb{S}^{3}$ (Figure %
\ref{trebol}), with group
\begin{equation*}
G(\mathbb{T})=\pi_{1}(\mathbb{S}^{3}\setminus K)=|a,b; aba=bab|
\end{equation*}
where $a$ and $b$ are meridian elements. The maps
\begin{equation}  \label{ewtrebol}
\begin{array}{cccc}
\omega_{\mathbb{T}}: & G(\mathbb{T})=\pi_{1}(\mathbb{S}^{3}\setminus \mathbb{T}) & \longrightarrow &
Isom(X_{(S,S)},Q),\quad S\neq 0 \\
& a & \to & a(\alpha ,\theta) \\
& b & \to & b(\alpha ,\theta)%
\end{array}%
\end{equation}
and
\begin{equation}  \label{erepresentationt}
\begin{array}{cccc}
\omega_{\mathbb{T}}: & G(\mathbb{T})=\pi_{1}(\mathbb{S}^{3}\setminus \mathbb{T}) & \longrightarrow &
Isom(X_{1},Q) \\
& a & \to & a_{t} \\
& b & \to & b_{t}%
\end{array}%
\end{equation}
are homomorphisms.
\end{proposition}

\begin{figure}[h]
\begin{center}
\epsfig{file=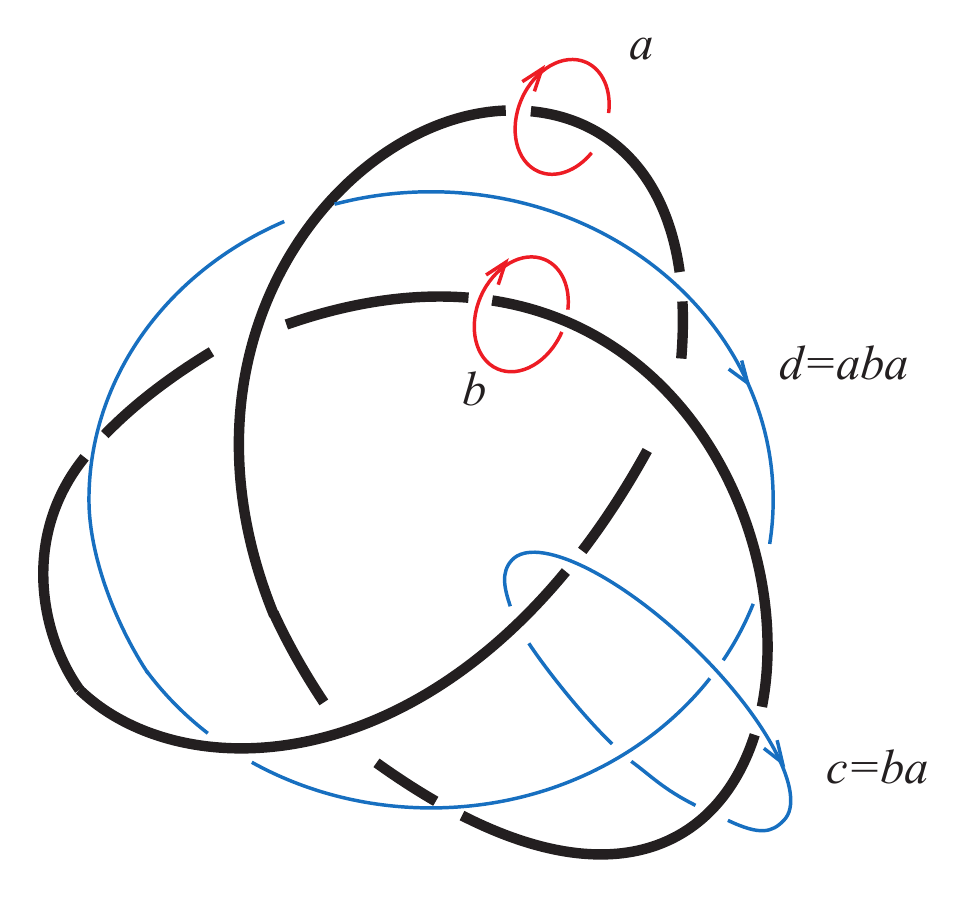,height=7cm}
\end{center}
\caption{The trefoil knot.}
\label{trebol}
\end{figure}

\begin{proof}
To prove that the map $\omega _{\mathbb{T}}$ in (\ref{ewtrebol}) is a homomorphism it
is enough to check the relator
\begin{equation}
a(\alpha ,\theta )b(\alpha ,\theta )a(\alpha ,\theta )=b(\alpha ,\theta
)a(\alpha ,\theta )b(\alpha ,\theta )  \label{eababab}
\end{equation}%
or equivalently
\begin{equation}
(MNM,R^{3})=(NMN,R^{3})  \label{eababab1}
\end{equation}%
A straightforward computation yields
\begin{equation*}
(NM)^{3}=(MNM)^{2}=-I_{2\times 2}
\end{equation*}%
Therefore
\begin{equation*}
MNMNMN=MNMMNM=-I_{2\times 2}\Longrightarrow NMN=MNM=\left[
\begin{array}{cc}
i & 0 \\
0 & -i
\end{array}%
\right]
\end{equation*}

The proof that the same is true in (\ref{erepresentationt})
\begin{equation*}
a_{t}b_{t}a_{t}=b_{t}a_{t}b_{t}
\end{equation*}%
can be done by direct computation using (\ref{eatbt}). In fact
\begin{equation*}
a_{t}b_{t}a_{t}=b_{t}a_{t}b_{t}=\left[
\begin{array}{cccc}
-1 & 0 & 0 & 0 \\
0 & -1 & 0 & 0 \\
0 & 0 & 1 & -2\sqrt{3}(6t+1) \\
0 & 0 & 0 & 1%
\end{array}%
\right]
\end{equation*}
\end{proof}

\begin{corollary}
For each pair $(\alpha ,\theta )$, $\alpha \in \left[ 0,\frac{\pi }{6}%
\right) \cup \left( \frac{\pi }{6},\frac{5\pi }{6}\right) $, there exits a
group of isometries in $(X_{(S,S)},Q)$ generated by $a(\alpha ,\theta )$ and
$b(\alpha ,\theta )$, which is the epimorphic image of the trefoil knot
group, such that it  projects by $p:X_{(S,S)}\longrightarrow D_{S}$, onto the holonomy
 of $(O,0|2,3,r)$, $r=\frac{\pi }{\alpha }$.

For $\alpha =\frac{\pi }{6}$ there exists infinite pairs of isometries $%
(a_{t},b_{t})$, $t\in \mathbb{R}$, of $(X_{1},Q)$ generating subgroups which
are the epimorphic image of the trefoil knot group, such that they all
projects by $p:X_{1}\longrightarrow \mathbb{C}$, onto the holonomy of $%
(O,0|2,3,6)$. \qed
\end{corollary}

\begin{summary}

The  holonomy (\ref{eholonomia}) of the geometric structure of
the 2-conemanifold $(O,0|2,3,r)$ is generated by $\omega (c)$ and $\omega
(d) $, where $c=ba$, $d=aba$ and
\begin{equation*}
\pi _{1}((O0|2,3,r))=|c,d;-|.
\end{equation*}

For each $r=\frac{\pi}{\alpha}$, $0\leq \alpha <\frac{5\pi}{6}$, $\alpha
\neq \frac{\pi}{6}$, the 2 conemanifold $(O,0|2,3,r)$ has a geometric
structure modeled in the disc $D_{S}$ of radius $1/\sqrt{|S|}$, where $S=%
\frac{1-2\sin \alpha}{1+2\sin \alpha}$. For $\alpha = \frac{\pi}{6}$, $D_{S}$
is $\mathbb{C}$ and the geometry is Euclidean.

For each pair $(\alpha ,\theta )$, $\alpha \in \left[ 0,\frac{\pi }{6}
\right) \cup \left( \frac{\pi }{6},\frac{5\pi }{6}\right) $, there exits a
group of isometries in $(X_{(S,S)},Q)$ generated by $a(\alpha ,\theta )$ and
$b(\alpha ,\theta )$, which is the epimorphic image of the trefoil knot
group, such that it  projects by $p:X_{(S,S)}%
\longrightarrow D_{S}$, onto the holonomy of $(O,0|2,3,r)$, $r=\frac{\pi }{%
\alpha }$.

For $\alpha =\frac{\pi }{6}$ there exists infinite pairs of isometries $%
(a_{t},b_{t})$, $t\in \mathbb{R}$, of $(X_{1},Q)$ generating subgroups which
are the epimorphic image of the trefoil knot group, such that it
projects by $p:X_{1}\longrightarrow \mathbb{C}$, on the holonomy of $%
(O,0|2,3,6)$.
\end{summary}

\section{Construction of the 3-conemanifolds with holonomy $\protect\omega%
_{\mathbb{T}}(G(\mathbb{T}))$}

Next we construct geometrically the 3-conemanifolds
whose holonomies are the above subgroups  $\omega _{\mathbb{T}}(G(\mathbb{T}))$ in (%
\ref{ewtrebol}) and (\ref{erepresentationt}).

Define
\begin{eqnarray}
c(\alpha ,\theta ) &=&b(\alpha ,\theta )\,a(\alpha ,\theta )  \notag \\
d(\alpha ,\theta ) &=&b(\alpha ,\theta )\,a(\alpha ,\theta )\,b(\alpha
,\theta )  \notag \\
c_{t} &=&b_{t}\,a_{t} \\
d_{t} &=&b_{t}\,a_{t}\,b_{t}.  \notag
\end{eqnarray}%
Then
\begin{eqnarray}
c(\alpha ,\theta ) &=&\left( \frac{1}{2}\left[
\begin{array}{cc}
1+2i\cos (\alpha ) & \sqrt{2\cos (2\alpha )-1} \\
\sqrt{2\cos (2\alpha )-1} & 1-2i\cos (\alpha )%
\end{array}%
\right] ,\left[
\begin{array}{cc}
e^{-2i\theta } & 0 \\
0 & e^{2i\theta }%
\end{array}%
\right] \right)  \notag \\
d(\alpha ,\theta ) &=&\left( \left[
\begin{array}{cc}
i & 0 \\
0 & -i%
\end{array}%
\right] ,\left[
\begin{array}{cc}
e^{-3i\theta } & 0 \\
0 & e^{3i\theta }%
\end{array}%
\right] \right)  \notag \\
c_{t} &=&\left[
\begin{array}{cccc}
-\frac{1}{2} & -\frac{\sqrt{3}}{2} & 0 & \frac{1}{2} \\
\frac{\sqrt{3}}{2} & -\frac{1}{2} & 0 & \frac{\sqrt{3}}{2} \\
\frac{\sqrt{3}}{2} & \frac{1}{2} & 1 & -\frac{1}{2}\sqrt{3}(16t+3) \\
0 & 0 & 0 & 1%
\end{array}%
\right] \\
d_{t} &=&\left[
\begin{array}{cccc}
-1 & 0 & 0 & 0 \\
0 & -1 & 0 & 0 \\
0 & 0 & 1 & -2\sqrt{3}(6t+1) \\
0 & 0 & 0 & 1%
\end{array}%
\right]  \notag
\end{eqnarray}%
Note that
\begin{eqnarray}
d(\alpha ,\theta )d(\alpha ,\theta ) &=&\left( \left[
\begin{array}{cc}
-1 & 0 \\
0 & -1%
\end{array}%
\right] ,\left[
\begin{array}{cc}
e^{-6i\theta } & 0 \\
0 & e^{6i\theta }%
\end{array}%
\right] \right) =  \notag \\
&=&\left( \left[
\begin{array}{cc}
1 & 0 \\
0 & 1%
\end{array}%
\right] ,\left[
\begin{array}{cc}
-e^{-6i\theta } & 0 \\
0 & -e^{6i\theta }%
\end{array}%
\right] \right)  \notag \\
&& \\
d_{t}d_{t} &=&\left[
\begin{array}{cccc}
1 & 0 & 0 & 0 \\
0 & 1 & 0 & 0 \\
0 & 0 & 1 & -4\sqrt{3}(6t+1) \\
0 & 0 & 0 & 1%
\end{array}%
\right]  \notag
\end{eqnarray}

Let $(\alpha ,\theta )$ be a pair such that $0\leq \alpha <\frac{\pi }{6}$.
The 2-conemanifold $(O0|2,3,\pi /\alpha )$ is hyperbolic and $S>0$. Consider
the hyperbolic triangle $\Delta $ in the interior of the Poincar\'{e} disc $%
D_{S}$ of radius $1/\sqrt{S}$, depicted in Figure \ref{fs23rhyper}. The
inverse image of $\Delta $ by the projection $p:X_{(S,S)}\longrightarrow
D_{S}$ is a fibred solid torus. Let $P=\Delta \times \mathbb{R}$ be its
universal cover.

The action of
\begin{equation*}
d(\alpha ,\theta )d(\alpha ,\theta )=\left( \left[
\begin{array}{cc}
1 & 0 \\
0 & 1%
\end{array}%
\right] ,\left[
\begin{array}{cc}
e^{-i(6\theta -\pi )} & 0 \\
0 & e^{i(6\theta -\pi )}%
\end{array}%
\right] \right)
\end{equation*}%
on $X_{(S,S)}$ induces a fibre preserving action on $P$. Suppose that each
fibre is oriented according to the action of $\mathbb{S}^{1}$ on $X_{(S,S)}$
(and of $\mathbb{R}$ on $P$).

The element $d(\alpha ,\theta )d(\alpha ,\theta )$ acts on $P$ by
translation of each fibre in the positive direction by a distance of $%
6\theta -\pi $.

Therefore the fundamental domain $D(\alpha ,\theta )$ for the action of $%
d(\alpha ,\theta )d(\alpha ,\theta )$ on $P$ is the part limited by the zero
level and the $6\theta -\pi $ level. Let us study the action of $c(\alpha
,\theta )$ and $d(\alpha ,\theta )$ on $D(\alpha ,\theta )$. Recall that
this action projects onto the action of $c$ and $d$ on $\Delta $.

\begin{figure}[h]
\begin{center}
\epsfig{file=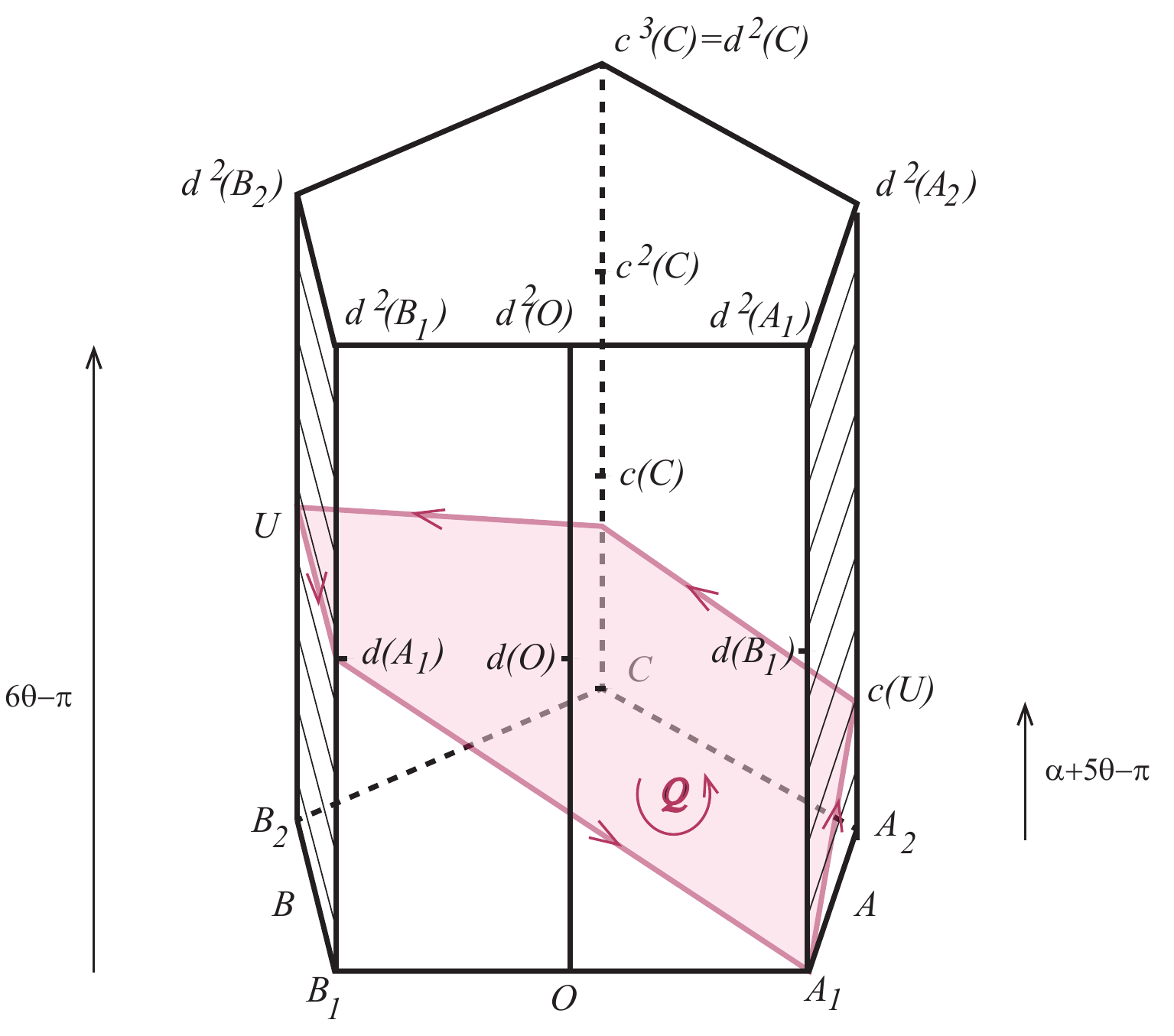,height=8cm}
\end{center}
\caption{The fundamental domain $D(\protect\alpha ,\protect\theta)$.}
\label{dalfateta}
\end{figure}

The two vertical side faces $A_{1},A_{2},d^{2}(A_{1}),d^{2}(A_{2})$ and $%
B_{1},B_{2},d^{2}(B_{1}),d^{2}(B_{2})$ of Figure \ref{dalfateta} are
horizontally fibred in such a way that the collapsing of these fibres yields
$D(\alpha ,\theta )$. The fibre $A=A_{1}A_{2}$ is sent to the fiber $%
U=d(\alpha ,\theta )(A)$ sitting on the face lying over $B=B_{1}B_{2}$, at
level $3\theta -\pi /2$. And $c(\alpha ,\theta )(U)$ belongs to the face
lying over $A=A_{1}A_{2}$, at level $\alpha +5\theta -\pi $:

In left-right notation
\begin{multline*}
d(\alpha ,\theta )=(NMN,R^{3})=\left( \left[
\begin{array}{cc}
i & 0 \\
0 & -i
\end{array}%
\right] ,\left[
\begin{array}{cc}
e^{-i3\theta } & 0 \\
0 & e^{i3\theta }
\end{array}%
\right] \right) = \\
=\left( \left[
\begin{array}{cc}
e^{i\pi /2} & 0 \\
0 & e^{-i\pi /2}
\end{array}%
\right] ,\left[
\begin{array}{cc}
e^{-i3\theta } & 0 \\
0 & e^{i3\theta }
\end{array}%
\right] \right)
\end{multline*}%
and the points $A\in D_{S}$ and $B\in D_{S}$ are respectively the
elements
\[
A=\frac{1}{\sqrt{1-S}}\left[
\begin{array}{cc}
1 & \sqrt{S} \\
\sqrt{S} & 1
\end{array}%
\right] \quad \text{and} \quad B=\frac{1}{\sqrt{1-S}}\left[
\begin{array}{cc}
1 & -\sqrt{S} \\
-\sqrt{S} & 1
\end{array}
\right].
\]
 Therefore, the following computation
\begin{multline*}
U=d(\alpha ,\theta )(A)=-\frac{1}{\sqrt{1-S}}\left[
\begin{array}{cc}
e^{i\pi /2} & 0 \\
0 & e^{-i\pi /2}
\end{array}%
\right] \left[
\begin{array}{cc}
1 & \sqrt{S} \\
\sqrt{S} & 1
\end{array}%
\right] \left[
\begin{array}{cc}
e^{-i3\theta } & 0 \\
0 & e^{i3\theta }
\end{array}%
\right] \\
=\frac{1}{\sqrt{1-S}}\left[
\begin{array}{cc}
e^{-i(3\theta -\pi /2)} & \sqrt{S}e^{i(3\theta +\pi /2)} \\
\sqrt{S}e^{-i(3\theta +\pi /2)} & e^{i((3\theta -\pi /2)}
\end{array}%
\right] \\
=\frac{1}{\sqrt{1-S}}\left[
\begin{array}{cc}
1 & -\sqrt{S} \\
-\sqrt{S} & 1
\end{array}%
\right] \left[
\begin{array}{cc}
e^{-i(3\theta -\pi /2)} & 0 \\
0 & e^{i((3\theta -\pi /2)}
\end{array}%
\right] \\
=B\left[
\begin{array}{cc}
e^{-i(3\theta -\pi /2)} & 0 \\
0 & e^{i(3\theta -\pi /2)}
\end{array}%
\right] .
\end{multline*}%
shows that the point $U=d(\alpha ,\theta )(A)$ is the point lying over the
point $B$ at the $(3\theta -\pi /2)$ level. Similarly, to compute $c(\alpha
,\theta )(U)$, consider the left-right notation of $c(\alpha ,\theta )\circ
d(\alpha ,\theta )$:
\begin{equation*}
c(\alpha ,\theta )\circ d(\alpha ,\theta )=(NM.NMN,R^{3}.R^{2})=(NMNMN,R^{5})
\end{equation*}%
where
\begin{multline*}
NMNMNM=-I_{2\times 2}\quad \Longrightarrow \\
NMNMN=-M^{-1}=-\frac{1}{1-S}\left[
\begin{array}{cc}
1 & \sqrt{S} \\
\sqrt{S} & 1
\end{array}%
\right] \left[
\begin{array}{cc}
e^{i-\alpha } & 0 \\
0 & e^{i\alpha }
\end{array}%
\right] \left[
\begin{array}{cc}
1 & -\sqrt{S} \\
-\sqrt{S} & 1
\end{array}%
\right]
\end{multline*}%
Then
\begin{multline*}
c(\alpha ,\theta )(U)=(c(\alpha ,\theta )\circ d(\alpha ,\theta ))(A)=-\frac{%
1}{\sqrt{1-S}}M^{-1}.\left[
\begin{array}{cc}
1 & \sqrt{S} \\
\sqrt{S} & 1
\end{array}%
\right] .R^{5}= \\
=-\frac{1}{\sqrt{1-S}}\left[
\begin{array}{cc}
1 & \sqrt{S} \\
\sqrt{S} & 1
\end{array}%
\right] \left[
\begin{array}{cc}
e^{-i\alpha } & 0 \\
0 & e^{i\alpha }
\end{array}%
\right] \left[
\begin{array}{cc}
e^{-i5\theta } & 0 \\
0 & e^{i5\theta }
\end{array}%
\right] = \\
=\frac{1}{\sqrt{1-S}}\left[
\begin{array}{cc}
1 & \sqrt{S} \\
\sqrt{S} & 1
\end{array}%
\right] \left[
\begin{array}{cc}
e^{-i(5\theta +\alpha -\pi )} & 0 \\
0 & e^{i(5\theta +\alpha -\pi )}
\end{array}%
\right] = \\
=A\left[
\begin{array}{cc}
e^{-i(5\theta +\alpha -\pi )} & 0 \\
0 & e^{i(5\theta +\alpha -\pi )}
\end{array}%
\right] .
\end{multline*}%

Then, the point $c(U)$ is the point lying over $A$ at the $(5\theta +\alpha
-\pi )$ level. See Figure\ref{dalfateta}.

The result of the following identifications

\begin{itemize}
\item the two horizontal faces: ($d^{2}(A_{1}),d^{2}(A_{2}),d^{2}(C),d^{2}(B{%
2}),d^{2}(B_{1})$) (upper face) and ($A_{1},A_{2},C,B_{2},B_{1}$) (bottom
face) by $d^{2}$;

\item the two front faces ($O,A_{1},d^{2}(A_{1}),d^{2}(O)$) and ($%
O,B_{1},d^{2}(B_{1}),d^{2}(O)$) by $d$;

\item and the two back faces ($C,A_{2},d^{2}(A_{2}),d^{2}(C)$) and ($%
O,B_{2},d^{2}(B_{2}),d^{2}(C)$) by $c$,
\end{itemize}
produce a Seifert manifold bounded by a torus with a foliation induced by
the horizontal fibres of the side faces $%
A_{1},A_{2},d^{2}(A_{1}),d^{2}(A_{2})$ and $%
B_{1},B_{2},d^{2}(B_{1}),d^{2}(B_{2})$. Topologically this is $\mathbb{S}%
^{3} $ minus an open solid torus. To determine its Seifert structure
consider an oriented section disc $Q$ whose orientation, followed
by the orientation of the general fibre $H$, gives the positive orientation
of $\mathbb{S}^{3}$. See Figure \ref{dalfateta}.

The identification of the two front faces, depicted in Figure \ref{dfrente}
is equivalent to collapsing a curve homologous to $2Q+H$, producing a
exceptional fibre of type $(2/1)$.
\begin{figure}[h]
\begin{center}
\epsfig{file=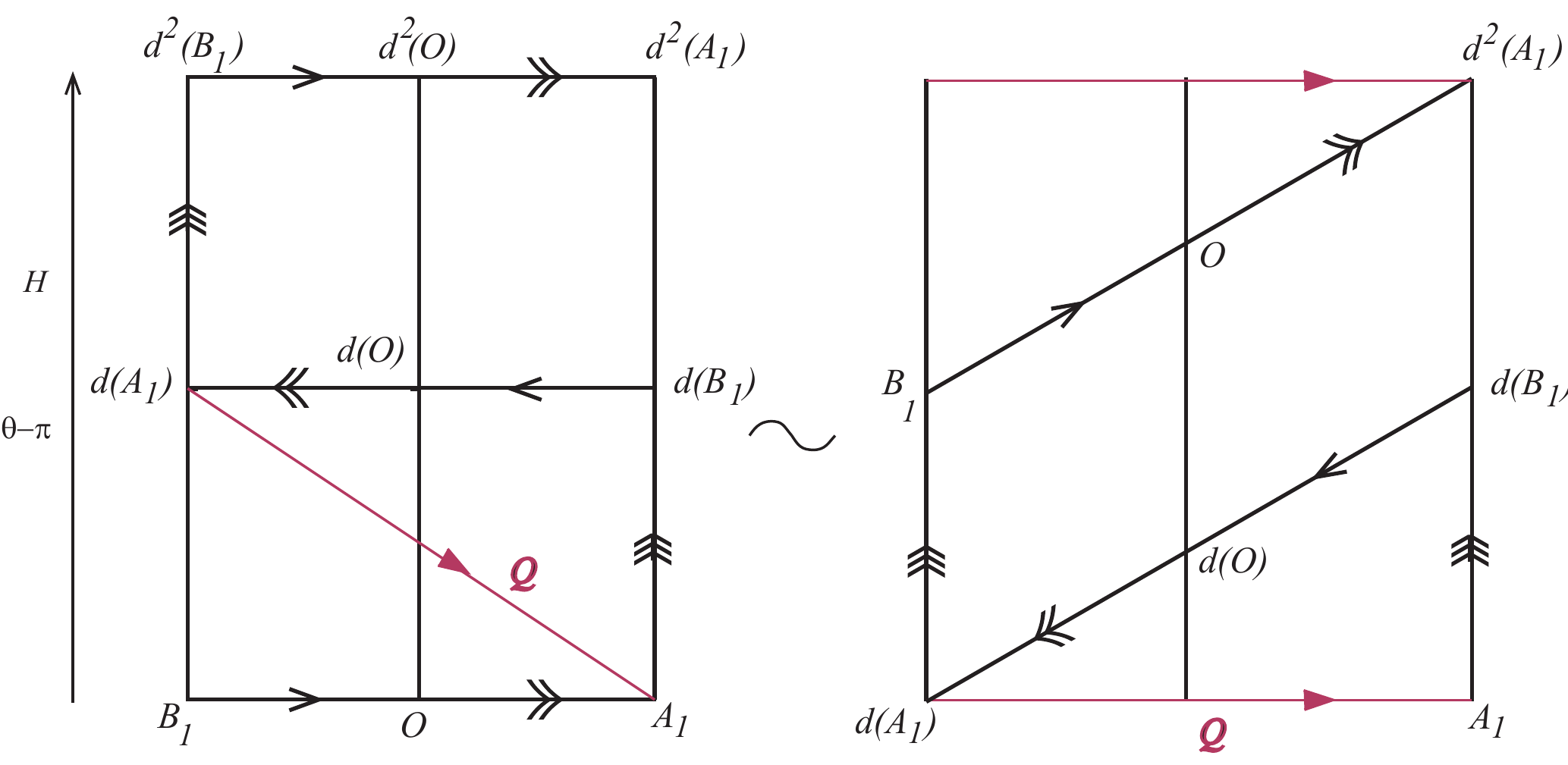,height=6cm}
\end{center}
\caption{The two front faces in $D(\protect\alpha ,\protect\theta )$.}
\label{dfrente}
\end{figure}

The identification of the two back faces, depicted in Figure \ref{ddetras}
is equivalent to collapsing a curve homologous to $3Q+H$, producing a
exceptional fibre of type $(3/1)$.

\begin{figure}[h]
\begin{center}
\epsfig{file=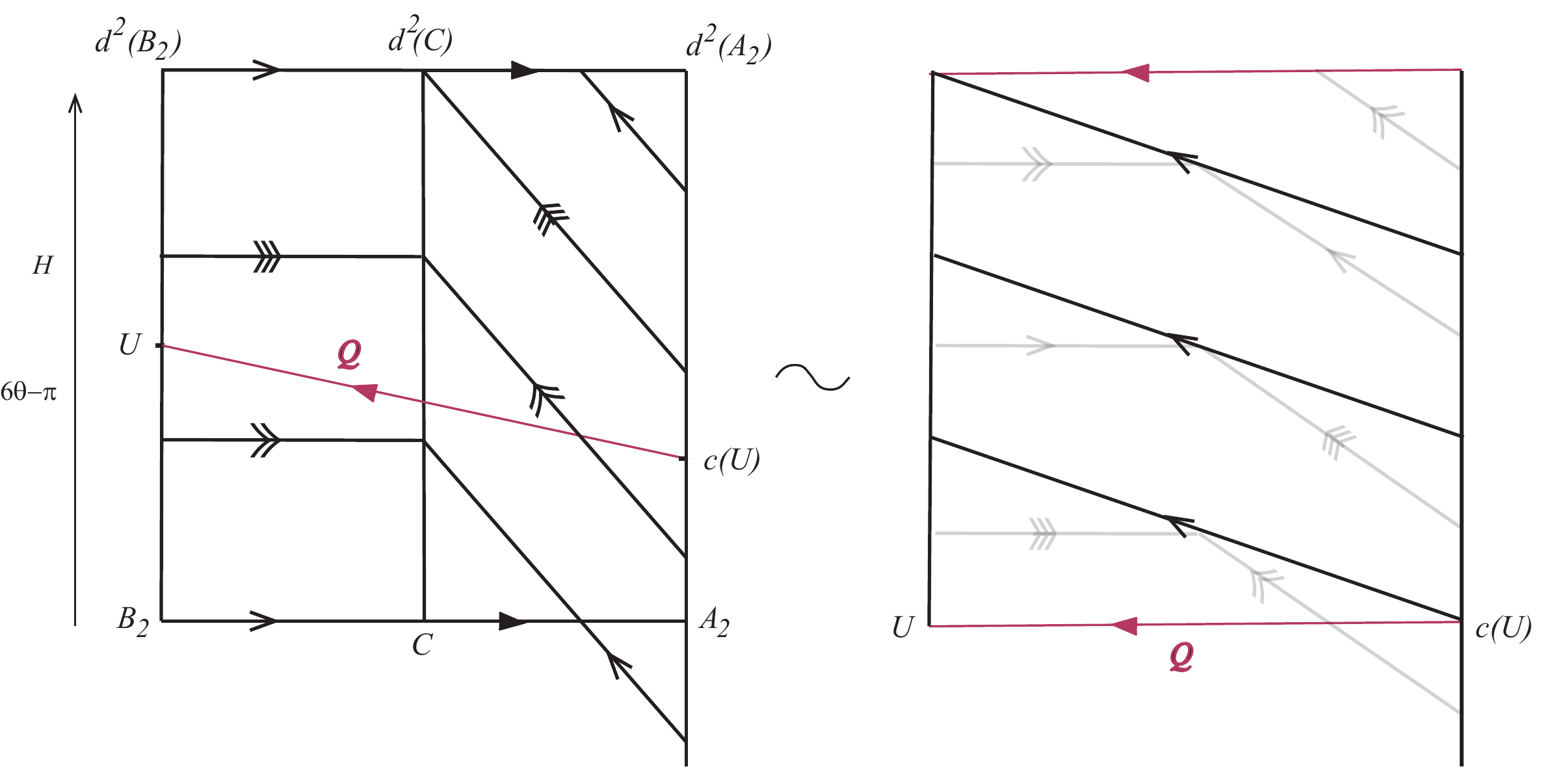,height=6cm}
\end{center}
\caption{The two back faces in $D(\protect\alpha ,\protect\theta)$.}
\label{ddetras}
\end{figure}

Therefore the Seifert structure after the identifications in $D(\alpha
,\theta )$ by $d^{2}$, $d$ and $c$ (before collapsing fibres in the torus
over $A\cup B$) is
\begin{equation}
(O\,o\,0|0;(2/1),(3/1))  \label{eseifert}
\end{equation}%
minus the neighbourhood of an ordinary fibre. The fibred torus which is the
boundary of the manifold after the above identifications is depicted in
Figure \ref{dtoro}. The slope of each fibre respect to the coordinates $Q$
and $H$ is $\lambda =\frac{\alpha +5\theta -\pi }{6\theta -\pi }$. If $%
\lambda $ is a rational number, the collapsing of the fibres in the torus is
equivalent to collapsing a curve homologous to $(6\theta -\pi
)Q-(\alpha +5\theta -\pi )H$.

\begin{figure}[h]
\begin{center}
\epsfig{file=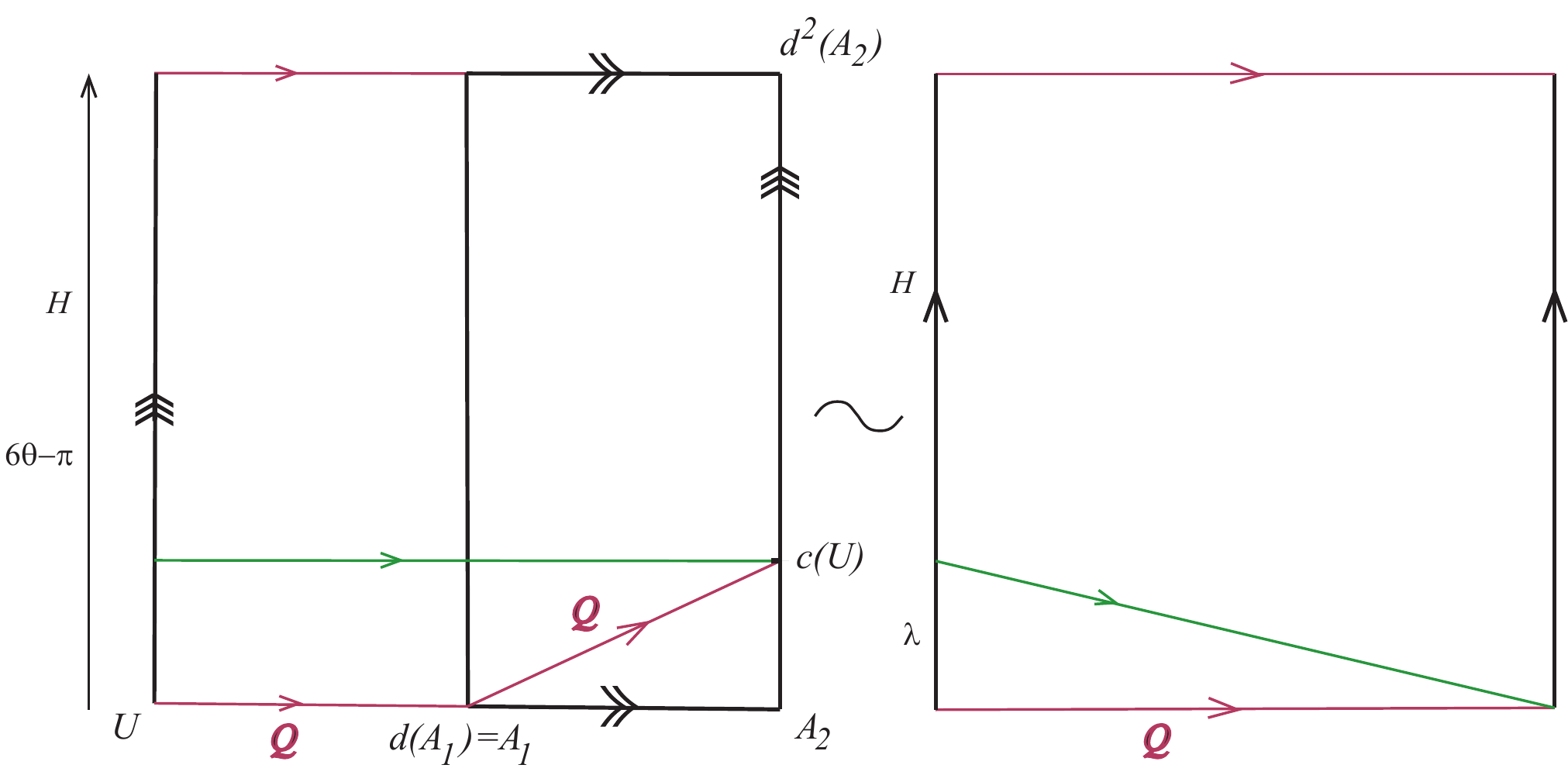,height=6cm}
\end{center}
\caption{The torus boundy.}
\label{dtoro}
\end{figure}

The resulting Seifert manifold is
\begin{eqnarray}
&&(O\,o\,0|0;(2/1),(3/1),((6\theta -\pi )/(\alpha +5\theta -\pi ))  \notag \\
&=&(O\,o\,0|-1;(2/1),(3/1),((6\theta -\pi )/(6\theta -\pi -\alpha -5\theta
+\pi )) \\
&=&(O\,o\,0|-1;(2/1),(3/1),((6\theta -\pi )/(\theta -\alpha ))  \notag
\end{eqnarray}%
This manifold is the result of $\left( \frac{6\theta -\pi }{\theta -\alpha }%
-6\right) $-surgery in the left-handed trefoil knot in $\mathbb{S}^{3}$.
This is because the surgery is always refered to the canonical longitude and
two parallel ordinary fibres in the Seifet structure $(O\,o%
\,0|0;(2/1),(3/1)) $ in $\mathbb{S}^{3}$ are two parallel left-handed
trefoil knots (can be considered in the same torus surface). Therefore, one
of them is a toroidal longitude $l_{t}$ for the other, and it is easy to
check (in Figure \ref{dlong}) that $l_{t}\,=\,l_{p}-3m\,=\,l_{c}-6m$, where $%
l_{p}$ is the pictorial longitude, $l_{c}$ is the canonical longitude and $m$
is the meridian of the left-handed trefoil knot.
\begin{figure}[h]
\begin{center}
\epsfig{file=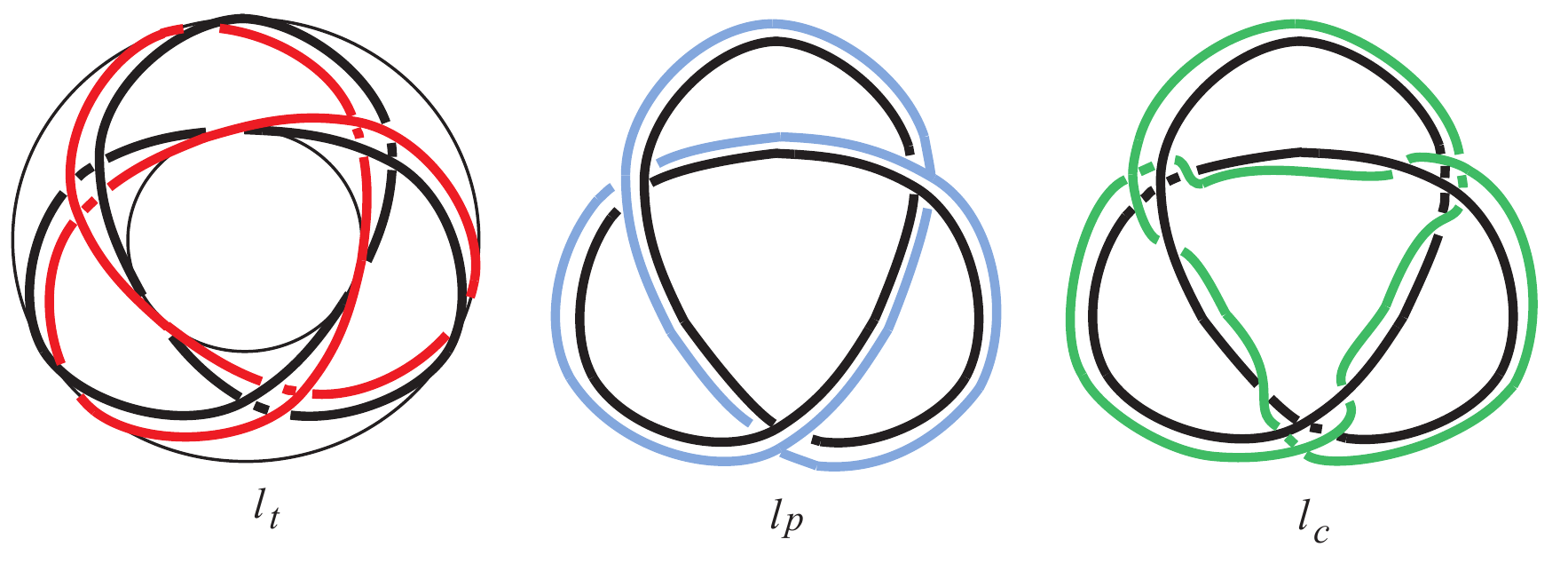,height=4cm}
\end{center}
\caption{The toroidal, pictorial and canonical longitudes.}
\label{dlong}
\end{figure}

\begin{theorem}
\label{tSeifertrebol} If $\frac{6\theta -\pi }{\theta -\alpha }$ is a
rational number, the quotient of $P$ by the group generated by $c(\alpha
,\theta )$ and $d(\alpha ,\theta )$ is the Seifert manifold
\begin{equation*}
(O\,o\,0|-1;(2/1),(3/1),((6\theta -\pi )/(\theta -\alpha ))
\end{equation*}%
which is the result of $\left( \frac{6\alpha -\pi }{\theta -\alpha }\right) $%
-surgery in the left-handed trefoil knot in $\mathbb{S}^{3}$. This manifold
has $\widetilde{S(2,\mathbb{R})}$ geometry for $0\leq \alpha <\frac{\pi }{6}$
and spherical geometry for $\frac{\pi }{6}<\alpha <\frac{5\pi }{6}$. The
conic angle $\beta $ is $2\alpha $ times the multiplicity $m$ of the
exceptional fibre, where $\frac{6\alpha -\pi }{\theta -\alpha }=\frac{m}{n}%
,\,\gcd (m,n)=1$. The singular points with angle $\beta $ form the core of
the surgery on the left-handed trefoil knot. This singular curve has length $%
\frac{6\theta -\pi }{m}$.
\end{theorem}

\begin{proof}
The first part of the theorem is already proved. Figure \ref{dtoro} shows
that if the slope $\frac{6\theta -\pi }{\theta -\alpha }$ is a rational
number, then the intersection of the surgery meridian with the fibre $H$ is
equal to the numerator of the reduced fraction $\frac{6\theta -\pi }{\theta
-\alpha }=\frac{m}{n}$, $\gcd (m,n)=1$. Each intersection point represent an
angle of $2\alpha $ because this is the angle of rotation around the points $%
A,\,B\in D_{S}$. Figure \ref{dtoro} shows also that its length is $\frac{%
6\theta -\pi }{m}$.
\end{proof}

\subsection{ Dehn surgery in the trefoil knot}

Consider the result of $p/q$ surgery in the left-handed trefoil knot $%
\mathbb{T}$, $\gcd (p,q)=1$. It is the Seifert manifold
\begin{equation}
(O\,o\,0|-1;(2/1),(3/1),(6+p/q))  \label{econemanifoldpq}
\end{equation}

Consider the 3-conemanifold $(\mathbb{T}_{p/q},r)$ whose underlying space is
the Seifert manifold in (\ref{econemanifoldpq}) with singular set the core
of the surgery (or equivalently, the exceptional fibre $(6+p/q)$) and with
valuation $r$. Let $\beta =2\pi /r$. Next we study the geometry possessed by
this conemanifold.

Suppose
\begin{equation*}
\left\{
\begin{array}{l}
\beta =\frac{2\pi }{r}=2\alpha (p+6q) \\
\\
p/q=\frac{6\alpha -\pi }{\theta -\alpha },\,q\neq 0;\quad p/q=\infty
,\,\alpha =\theta%
\end{array}%
\right. , \\
\end{equation*}
then
\begin{equation}  \label{ealfayteta}
\left\{
\begin{array}{l}
\alpha =\frac{\pi}{r(p+6q)}=\frac{\beta}{2(p+6q)} \\
\\
\theta=\alpha+\frac{q}{p}(6\alpha -\pi )=\pi \left(\frac{1 }{pr}-\frac{q}{p}%
\right)=\frac{\beta -2\pi q}{2p}%
\end{array}
\;\quad \Longrightarrow \quad p\neq 0. \right.
\end{equation}

If $p\neq 0$, by Theorem \ref{tSeifertrebol}, the conemanifold $(\mathbb{T}_{p/q},r)$ has spherical geometry for
\begin{equation*}
\frac{\pi}{6}<\alpha <\frac{5\pi}{6}\Longleftrightarrow \frac{6}{5|p+6q|}%
<|r|<\frac{6}{|p+6q|}\Longleftrightarrow \frac{\pi}{3}|p+6q|<|\beta |<\frac{%
5\pi}{3}|p+6q|
\end{equation*}
and $\widetilde{SL(2,\mathbb{R})}$ geometry for
\begin{equation*}
0\leq \alpha <\frac{\pi}{6}\Longleftrightarrow \frac{6}{|p+6q|}<|r|\leq
\infty \Longleftrightarrow 0\leq |\beta |<\frac{\pi}{3}|p+6q|.
\end{equation*}

For the limit case $\alpha \to \frac{\pi}{6}$ the generators $a_{t}$ and $%
b_{t}$ in (\ref{eatbt}) are well defined because $\theta=\alpha+\frac{q}{p}%
(6\alpha -\pi )$.
\begin{eqnarray}  \label{eaqpbqp}
a_{q/p} =& \left[
\begin{array}{cccc}
\frac{1}{2} & -\frac{\sqrt{3}}{2} & 0 & \frac{1}{2} \\
\frac{\sqrt{3}}{2} & \frac{1}{2} & 0 & -\frac{\sqrt{3}}{2} \\
\frac{\sqrt{3}}{2} & -\frac{1}{2} & 1 & -\frac{1}{2} \sqrt{3} (8 \frac{q}{p}%
+1) \\
0 & 0 & 0 & 1%
\end{array}
\right]  \notag \\
\\
b_{q/p} =& \left[
\begin{array}{cccc}
\frac{1}{2} & -\frac{\sqrt{3}}{2} & 0 & -\frac{1}{2} \\
\frac{\sqrt{3}}{2} & \frac{1}{2} & 0 & \frac{\sqrt{3}}{2} \\
-\frac{\sqrt{3}}{2} & \frac{1}{2} & 1 & -\frac{1}{2} \sqrt{3} (8 \frac{q}{p}%
+1) \\
0 & 0 & 0 & 1%
\end{array}
\right]  \notag
\end{eqnarray}
\begin{figure}[h]
\begin{center}
\epsfig{file=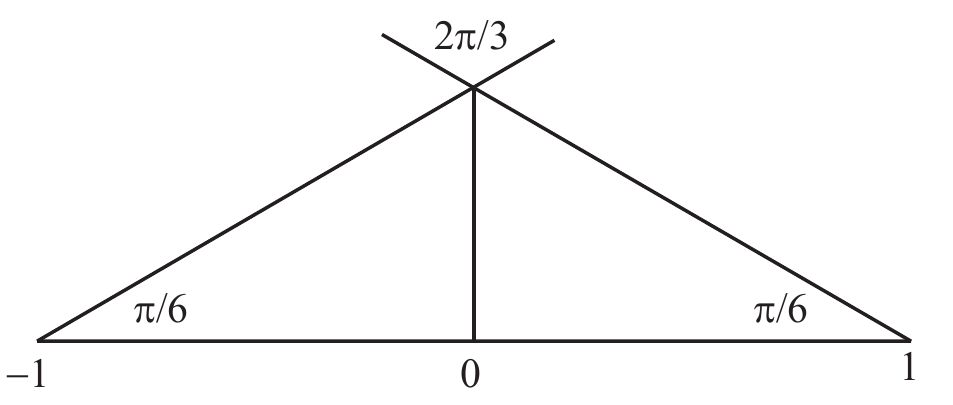,height=2.5cm}
\end{center}
\caption{The Euclidean triangle $\Delta$.}
\label{ftrianeu}
\end{figure}
\begin{figure}[h]
\begin{center}
\epsfig{file=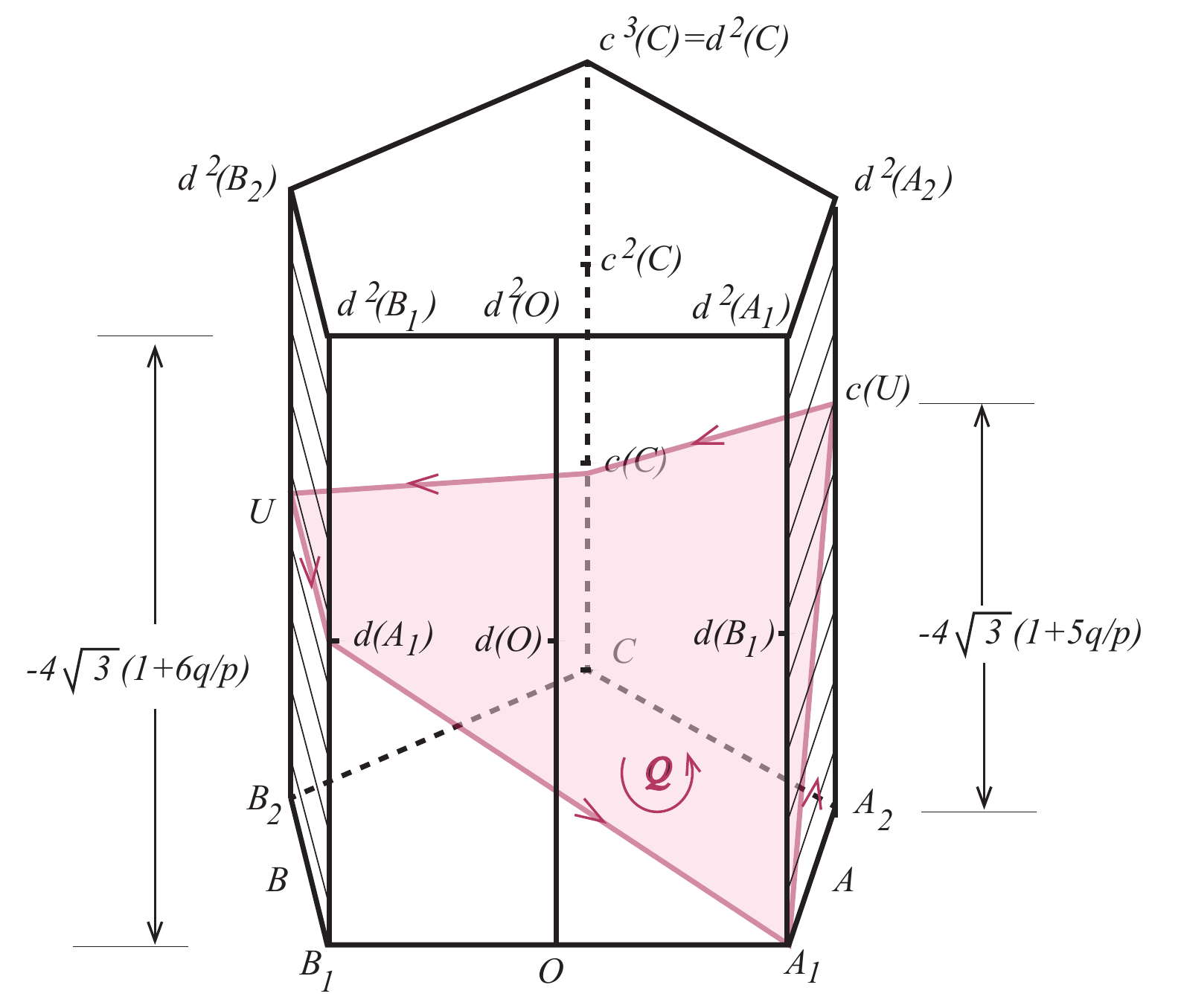,height=8cm}
\end{center}
\caption{The fundamental domain in the Nil geometry  $(X_{1},Q)$.}
\label{dalfatetapq}
\end{figure}

We can obtain explicitly the Nil geometry $(X_{1},Q)$  in $(\mathbb{T}_{p/q},%
\frac{6}{p+6q})$. Let $\Delta $ be the Euclidean triangle of Figure \ref%
{ftrianeu} . Because
\begin{equation*}
d_{q/p}d_{q/p}=\left[
\begin{array}{cccc}
1 & 0 & 0 & 0 \\
0 & 1 & 0 & 0 \\
0 & 0 & 1 & -4\sqrt{3}(6\frac{q}{p}+1) \\
0 & 0 & 0 & 1%
\end{array}%
\right]
\end{equation*}%
we can take the right prism with base $\Delta $ minus a small neighborhood
of vertex $1$ and $-1$ and height $-4\sqrt{3}(6\frac{q}{p}+1)$ (Figure \ref%
{dalfatetapq}) as fundamental domain for the action of the group of
isometries in the Nil geometry $(X_{1},Q)$, generated by $a_{q/p}$ and $b_{q/p}$ (or $c_{q/p}$
and $d_{q/p}$). The element $d_{q/p}^{2}$ identifies the two horizontal
faces (bases) by translation. The element
\begin{equation*}
d_{q/p}=\left[
\begin{array}{cccc}
-1 & 0 & 0 & 0 \\
0 & -1 & 0 & 0 \\
0 & 0 & 1 & -2\sqrt{3}(6\frac{q}{p}+1) \\
0 & 0 & 0 & 1%
\end{array}%
\right]
\end{equation*}%
identifies the two halfs of the front face producing a exceptional fibre $%
(2/1)$ as in the general case. The element
\begin{equation*}
c_{q/p}=\left[
\begin{array}{cccc}
-\frac{1}{2} & -\frac{\sqrt{3}}{2} & 0 & \frac{1}{2} \\
\frac{\sqrt{3}}{2} & -\frac{1}{2} & 0 & \frac{\sqrt{3}}{2} \\
\frac{\sqrt{3}}{2} & \frac{1}{2} & 1 & -\frac{1}{2}\sqrt{3}(16\frac{q}{p}+3)
\\
0 & 0 & 0 & 1%
\end{array}%
\right]
\end{equation*}%
\begin{figure}[h]
\begin{center}
\epsfig{file=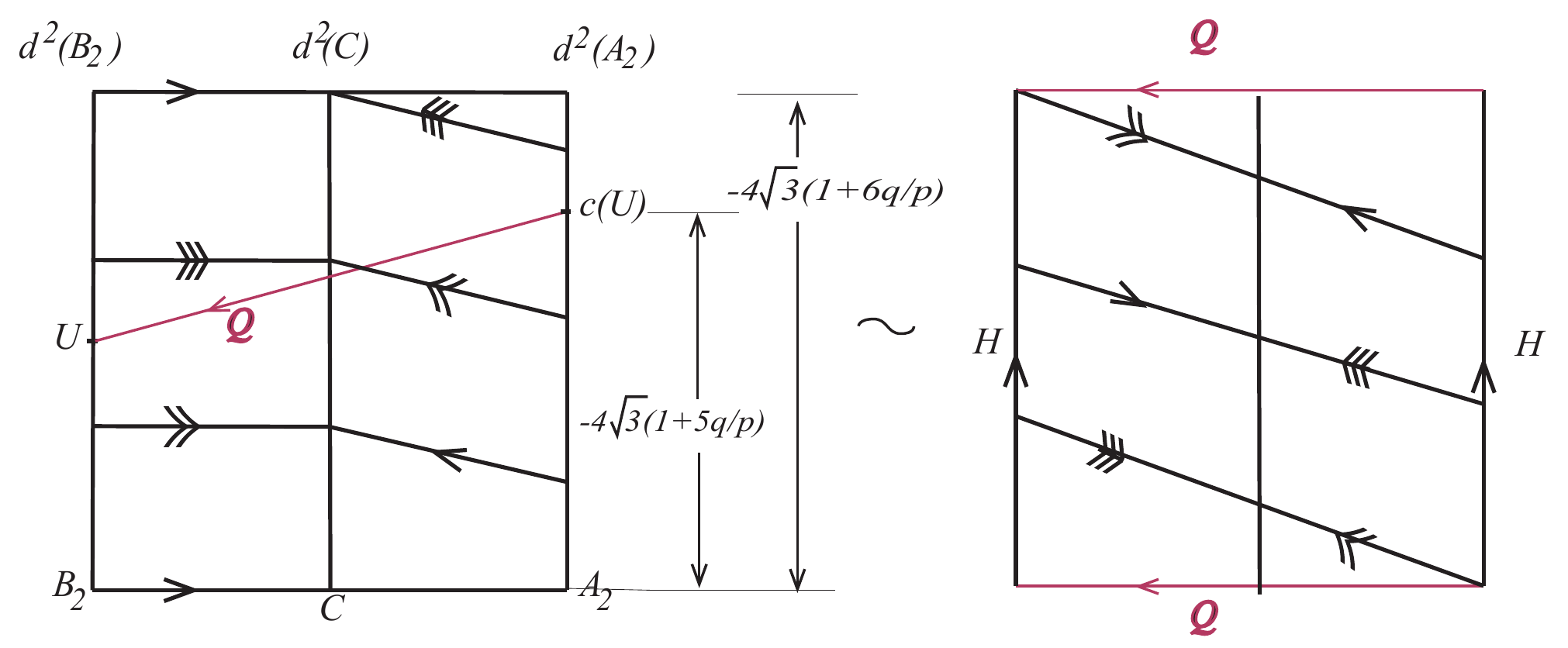,height=5cm}
\end{center}
\caption{The exceptional fibre (3/1).}
\label{detraspq}
\end{figure}
identifies the two back faces as it is shown in Figure \ref{detraspq},
because
\begin{equation*}
(c_{q/p}\circ \, b_{q/p})(A)=c_{q/p}\left(b_{q/p}\left(
\begin{array}{c}
1 \\
0 \\
0 \\
1 \\
\end{array}%
\right) \right)=c_{q/p}\left(
\begin{array}{c}
-1 \\
0 \\
-2\sqrt{3}(6\frac{q}{p}+1) \\
1 \\
\end{array}%
\right) =\left(
\begin{array}{c}
1 \\
0 \\
-4\sqrt{3}(5\frac{q}{p}+1) \\
1 \\
\end{array}%
\right)
\end{equation*}%
producing a $(3/1)$ exceptional fibre.
\begin{figure}[h]
\begin{center}
\epsfig{file=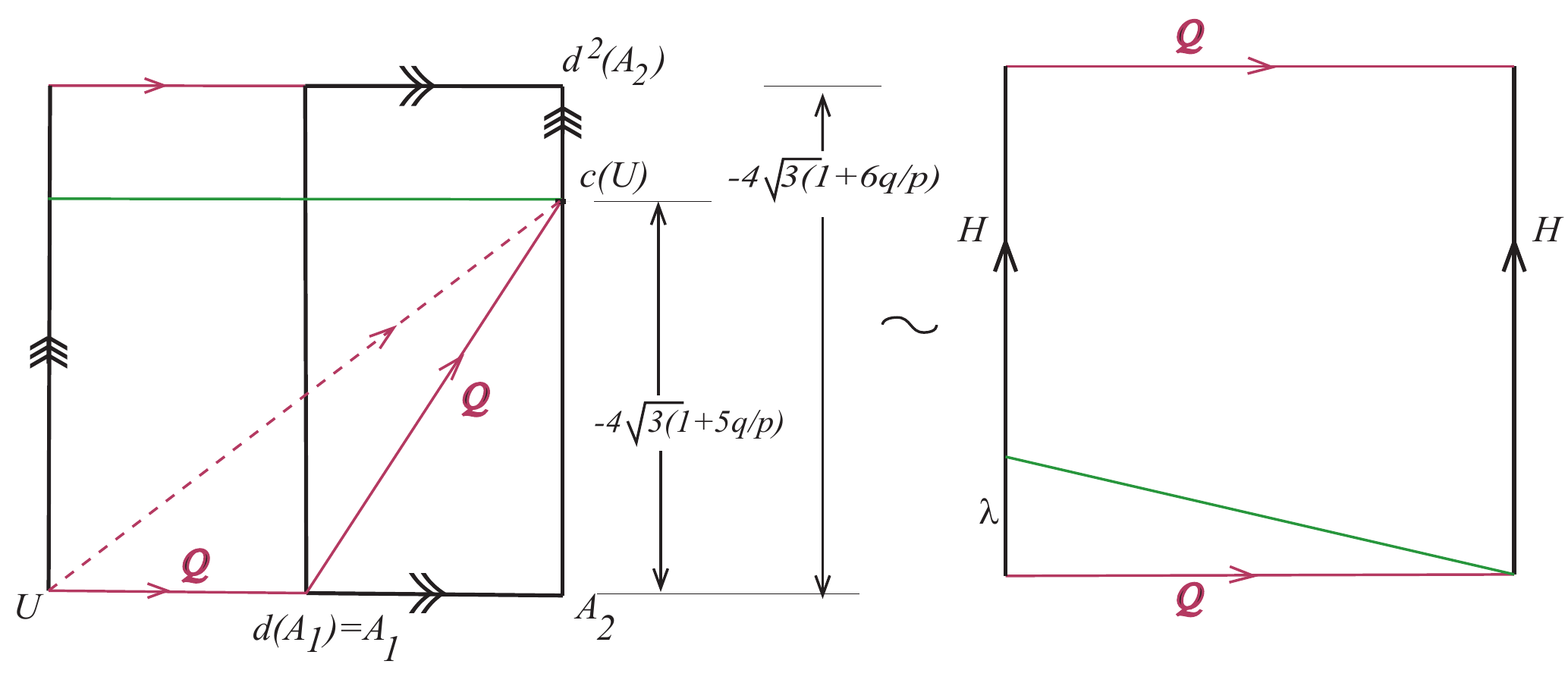,height=5cm}
\end{center}
\caption{The exceptional fibre $(\frac{p}{q}+6)$.}
\label{dtoropq}
\end{figure}

The exceptional fibre coming from the collapsing of the leaves of the
foliation in the boundary torus, depicted in Figure \ref{dtoropq}, is
\begin{equation*}
\frac{4\sqrt{3}(1+6\frac{q}{p})}{-4\sqrt{3}(1+5\frac{q}{p})}=\frac{1+6\frac{q%
}{p}}{-(1+5\frac{q}{p})}.
\end{equation*}

The Seifert manifold is
\begin{equation}
\begin{array}{l}
(O\,o\,0|0;(2/1),(3/1),((1+6\frac{q}{p})/(-1-5\frac{q}{p})))= \\
(O\,o\,0|-1;(2/1),(3/1),((1+6\frac{q}{p})/(1+6\frac{q}{p}-1-5\frac{q}{p})))=
\\
(O\,o\,0|-1;(2/1),(3/1),(6+\frac{q}{p})),%
\end{array}
\label{eseifertnil}
\end{equation}%
which coincides with (\ref{econemanifoldpq}).

The length $\lambda $ of the singular curve (core of the surgery) is given
in Theorem \ref{tSeifertrebol}, namely $\lambda =\frac{6\theta -\pi }{m}$,
where $m$ is the multiplicity of the exceptional fibre $\frac{p}{q}+6=\frac{%
p+6q}{q}=\frac{m}{n}$, $\gcd (m,n)=1$. Therefore, by (\ref{ealfayteta})
\begin{equation*}
\lambda =\frac{6\theta -\pi }{p+6q}=\frac{6\pi \left( \frac{1}{pr}-\frac{q}{p%
}\right) -\pi }{p+6q}=\frac{6\pi \left( 1-qr\right) -pr\pi }{pr(p+6q)}=\frac{%
6\pi }{pr(p+6q)}-\frac{\pi }{p}.
\end{equation*}

Summarizing:

\begin{theorem}
The conemanifold $(\mathbb{T}_{p/q},r)$, $p\neq 0$, has spherical geometry
for
\begin{equation*}
\frac{6}{5|p+6q|}<|r|<\frac{6}{|p+6q|},
\end{equation*}
Nil $X_{1}$ geometry for $|r|=\frac{6}{|p+6q|}$ and $\widetilde{SL(2,\mathbb{R})}
$ geometry for
\begin{equation*}
\frac{6}{|p+6q|}<|r|\leq \infty
\end{equation*}
The holonomy is generated by $a(\alpha ,\theta )$ (\ref{eaAij}) and $%
b(\alpha ,\theta )$ (\ref{ebBij}) where
\begin{equation*}
\alpha =\frac{\pi}{r(p+6q)}\qquad \theta=\pi \left(\frac{1 }{pr}-\frac{q}{p}%
\right)
\end{equation*}
when $|r|\neq \frac{6}{|p+6q|}$ and by $a_{q/p}$ and $b_{q/p}$ (\ref{eaqpbqp}%
) if $|r|= \frac{6}{|p+6q|}$. The length of the singular knot is
\begin{equation}  \label{elongitud}
\lambda = \frac{6\pi}{pr(p+6q)}-\frac{\pi}{p}
\end{equation}
\qed
\end{theorem}

The following result adresses the remaining case $p=0$.

\begin{theorem}
The conemanifold $(\mathbb{T}_{0},r)$, has Euclidean geometry for $r=1$; $%
H^{2}\times \mathbb{R}$ geometry for
\begin{equation*}
1<|r|\leq \infty ;
\end{equation*}%
and $S^{2}\times \mathbb{R}$ geometry for
\begin{equation*}
\frac{1}{5}<|r|<1.
\end{equation*}
\end{theorem}

\begin{proof}
The trefoil knot is a fibred knot. The complement $C(\mathbb{T})$ of the trefoil knot
is obtained from $F_{1,1}\times \lbrack 0,1]$, where the punctured torus $%
F_{1,1}$ is a Seifert surface of the knot, by the identification $%
(x,0)=((h(x),1)$, where $h:F_{1,1}\longrightarrow F_{1,1}$ is an orientation
preserving cyclic homeomorphism of order 6. The $0$-surgery on the knot
consists in pasting a solid torus to $C(\mathbb{T})$ so that the boundary of the
meridian disc in the torus is identified with the boundary of the Seifert
surface. Therefore $\mathbb{T}_{0}=F_{1}\times \lbrack
0,1]/(x,0)=((h^{\prime }(x),1)$, where $h^{\prime }:F_{1}\longrightarrow
F_{1}$, extension of $h$, is a orientation preserving cyclic homeomorphism
of order $6$.
\begin{figure}[h]
\begin{center}
\epsfig{file=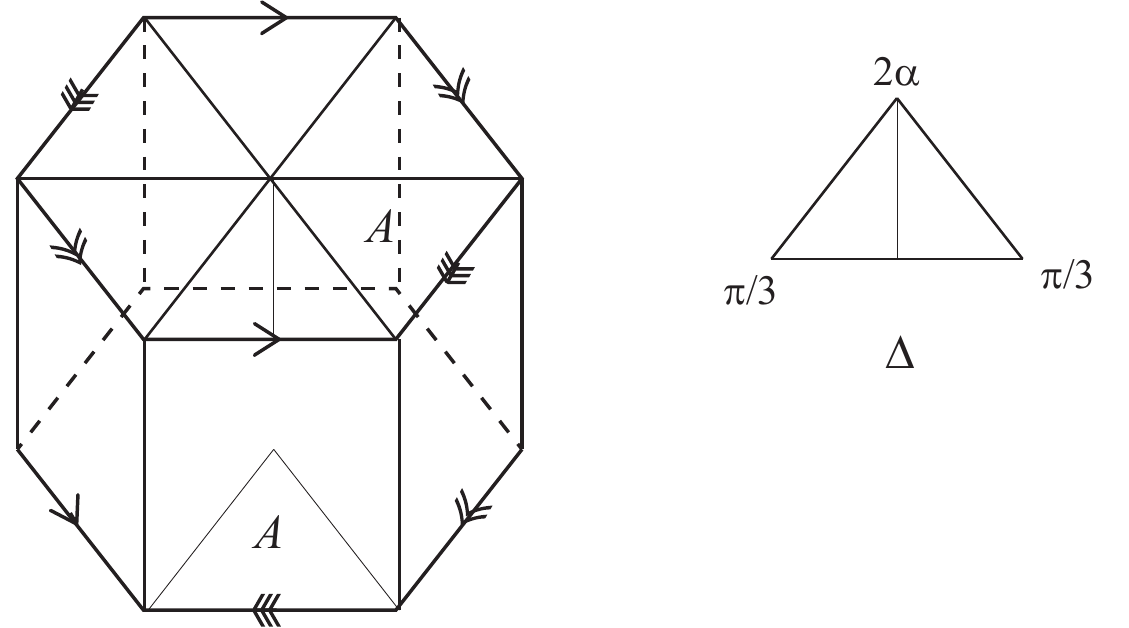,height=6cm}
\end{center}
\caption{The manifold $\mathbb{T}_{0}$.}
\label{fcirugia0}
\end{figure}

Figure \ref{fcirugia0} shows this manifold by identifications in a hexagonal
right prism, where the base is the union of $6$ isosceles triangles $\Delta $
with angles $\pi /3$, $\pi /3$ and $2\alpha =2\pi /6r$. If the angle $\alpha
$ is equal to $\pi /6$ ($r=1$) the hexagon lies on the Euclidean plane $E^{2}
$ and the prism lies on $E^{2}\times \mathbb{R}$; if $\alpha <\pi /6$ ($%
1<|r|\leq \infty $) the hexagon is a 2-conemanifold in
the hyperbolic plane $H^{2}$; and finally if $\pi /6<\alpha <5\pi
/6$ ($\frac{1}{5}<|r|<1$) the hexagon is a 2-conemanifold in the 2-sphere $S^{2}$.
\end{proof}

One way to resume the different geometries in $(\mathbb{T}_{p/q},r)$,
according to the different values of $p/q$ and $r$ is by creating a plot in $%
\mathbb{R}^{2}$ as follows.

\begin{definition}
The \emph{lower limit} of sphericity $l_{i}$ of the conemanifold $\mathbb{T}%
_{p/q}$ is equal to $\frac{6}{p+6q}$, and the \emph{upper limit} of
sphericity $l_{s}$ is equal to $\frac{6}{5(p+6q)}$.
\end{definition}

\begin{figure}[h]
\begin{center}
\epsfig{file=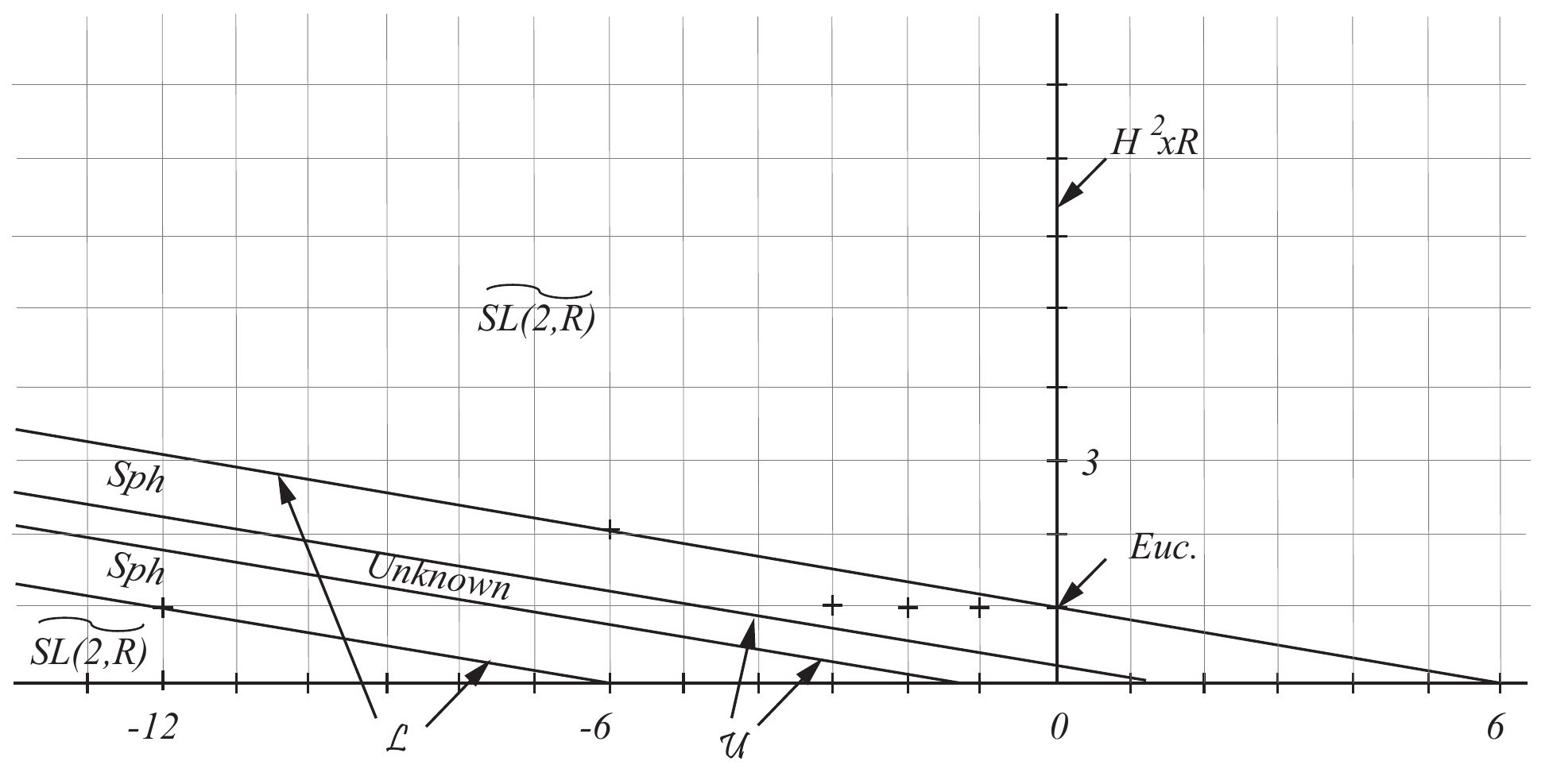,height=6cm}
\end{center}
\caption{The plot $\mathcal{P}_{1}$.}
\label{fplot1}
\end{figure}

In the plot $\mathcal{P}_{1}$ of Figure \ref{fplot1} the point $(rp,rq)$, with
integer coordinates, where $gcd(p,q)=1$ and $q>0$, $r>0$ represents the surgery $p/q$ in the left-handed trefoil
knot with conic angle $\beta =2\pi /r$, that is  the cone manifold $(\mathbb{%
T}_{p/q},r)$.

Let $\varepsilon$ be the sign of $p+6q$. The set of lower limits is
\begin{equation*}
\mathcal{L}=\left\{ (x,y)\in \mathbb{R}^{2}:\;x=\frac{\varepsilon 6p}{p+6q}%
,\;y=\frac{\varepsilon 6q}{p+6q}\right\}
\end{equation*}
and the set of upper limits is
\begin{equation*}
\mathcal{U}=\left\{ (x,y)\in \mathbb{R}^{2}:\;x=\frac{\varepsilon 6p}{5(p+6q)%
},\;y=\frac{\varepsilon 6q}{5(p+6q)}\right\}.
\end{equation*}
Both sets constitute a pair of straight lines depicted in Figure \ref{fplot1} and
they divide the plane in regions with different geometries.
\begin{equation*}
\begin{array}{c}
\mathcal{L}= \left\{ (x,y)\in \mathbb{R}^{2}: x+6y=6\varepsilon \right\} \\
\mathcal{U}= \left\{ (x,y)\in \mathbb{R}^{2}: x+6y=\frac{6}{5}\varepsilon
\right\}%
\end{array}%
\end{equation*}

Points in $\mathcal{L}$ represent conemanifolds with Nil $X_{1}$ geometry. We do
not know if the points in the region limited by the two straight lines in $%
\mathcal{U}$, including both lines, represent any geometric structure on the conemanifold $\mathbb{T}_{p/9}$,
 compatible with its as Seifert structure (fibres must be geodesics).
An analogous plot is contained in \cite{K1984}.

\subsection{Volume of the cone-manifold}

To compute the volume of a family of spherical or hyperbolic cone manifold a
normalized version of the metric should be used. The normalization for $%
X_{(S,S)}$, $S\neq 0$, consists in considering $|S|=1$. Then the metric
matrix in Seifert coordinates, (\ref{maqSc}), for the normalized $X_{(S,S)}$%
, $S>0$, is the following

\begin{equation}
Q_{1}((\mu ,\nu ,\theta )=\left[
\begin{array}{ccc}
\frac{\nu ^{2}+1}{\left( 1-\left( \mu ^{2}+\nu ^{2}\right) \right) ^{2}} & -%
\frac{\mu \nu }{\left( 1-\left( \mu ^{2}+\nu ^{2}\right) \right) ^{2}} &
\frac{\nu }{1-\left( \mu ^{2}+\nu ^{2}\right) } \\
&  &  \\
\ast  & \frac{\mu ^{2}+1}{\left( 1-\left( \mu ^{2}+\nu ^{2}\right) \right)
^{2}} & -\frac{\mu }{1-\left( \mu ^{2}+\nu ^{2}\right) } \\
&  &  \\
\ast  & \ast  & 1%
\end{array}%
\right] ,  \label{maqSc1}
\end{equation}%
and for the normalized $X_{(S,S)}$, when $S<0$, is
\begin{equation}
Q_{-1}((\mu ,\nu ,\theta )=\left[
\begin{array}{ccc}
\frac{\nu ^{2}+1}{\left( 1+\left( \mu ^{2}+\nu ^{2}\right) \right) ^{2}} & -%
\frac{\mu \nu }{\left( 1+\left( \mu ^{2}+\nu ^{2}\right) \right) ^{2}} & -%
\frac{\nu }{1+\left( \mu ^{2}+\nu ^{2}\right) } \\
&  &  \\
\ast  & \frac{\mu ^{2}+1}{\left( 1+\left( \mu ^{2}+\nu ^{2}\right) \right)
^{2}} & \frac{\mu }{1+\left( \mu ^{2}+\nu ^{2}\right) } \\
&  &  \\
\ast  & \ast  & 1%
\end{array}%
\right] .  \label{maqScm1}
\end{equation}%
Their determinants are respectively
\begin{equation*}
D(Q_{1})=\frac{1}{\left( 1-\left( \mu ^{2}+\nu ^{2}\right) \right) ^{4}}%
,\quad \quad D(Q_{-1})=\frac{1}{\left( 1+\left( \mu ^{2}+\nu ^{2}\right)
\right) ^{4}}
\end{equation*}

Suppose $S>0$. The volume form in the Seifert coordinates for the normalized
$X_{(S,S)}$, ($X_{(1,1)}$), is
\begin{equation*}
dv=\sqrt{|D(Q_{1})|}d\mu d\nu d\zeta =\frac{1}{\left( 1-\left( \mu ^{2}+\nu
^{2}\right) \right) ^{4}}d\mu d\nu d\zeta
\end{equation*}%
Therefore
\begin{eqnarray*}
V(\mathbb{T}_{p/q},r) &=&\int_{D(\alpha ,\theta )}dv=\int_{D(\alpha ,\theta
)}\frac{1}{\left( 1-\left( \mu ^{2}+\nu ^{2}\right) \right) ^{4}}d\mu d\nu
d\zeta  \\
&=&\int_{0}^{\zeta _{0}}d\zeta \int_{\Delta }\frac{1}{\left( 1-\left( \mu
^{2}+\nu ^{2}\right) \right) ^{4}}d\mu d\nu =\zeta _{0}\times \frac{1}{4}%
\text{Area of }\Delta
\end{eqnarray*}%
where $\zeta _{0}$ and $\Delta $ are, respectively, the height and the base
of $D(\alpha ,\theta )$. The geometry in the base is hyperbolic. By
construction of the fundamental domain $D(\alpha ,\theta )$, the base is a
hyperbolic triangle with angles $2\pi /3,\alpha ,\alpha $, and the height $%
\zeta _{0}$ is $|6\theta -\pi |$, where (\ref{ealfayteta})
\begin{equation*}
\alpha =\frac{\pi }{r(p+6q)}\qquad \theta =\frac{\pi (1-qr)}{pr}
\end{equation*}

Therefore
\begin{equation*}
\text{Area of }\Delta = \pi -2\pi /3 -2 \alpha =\pi /3-\frac{2\pi}{r(p+6q)}=
\frac{\pi r(p+6q)-6\pi}{3r(p+6q)}
\end{equation*}

\begin{equation*}
V(\mathbb{T}_{p/q},r)=\left|\frac{6\pi (1-qr)}{pr}-\pi\right|\left(\frac{%
\pi(r(p+6q)-6\pi )}{3r(p+6q)}\right)\frac{1}{4}=\left|- \frac{\pi ^2(p r+6 q r-6)^2%
}{12 p r^2 (p+6 q)}\right|.
\end{equation*}

Suppose $S<0$. The volume form in the Seifert coordinates for the normalized
$X_{(S,S)}$, ($X_{(-1,-1)}$), is
\begin{equation*}
dv=\sqrt{|D(Q_{-1})|}d\mu d\nu d\zeta =\frac{1}{\left( 1+\left( \mu ^{2}+\nu
^{2}\right)\right)^4}d\mu d\nu d\zeta
\end{equation*}
Therefore
\begin{eqnarray*}
V(\mathbb{T}_{p/q},r)&=& \int_{D(\alpha ,\theta )}dv= \int_{D(\alpha ,\theta
)}\frac{1}{\left( 1+\left( \mu ^{2}+\nu ^{2}\right)\right)^4}d\mu d\nu d\zeta
\\
&=& \int_{0}^{\zeta _{0}}d\zeta \int_{\Delta } \frac{1}{\left( 1+\left( \mu
^{2}+\nu ^{2}\right)\right)^4}d\mu d\nu =\zeta _{0}\times \frac{1}{4}\text{%
Area of }\Delta
\end{eqnarray*}
where $\zeta _{0}$ and $\Delta$ are respectively the height and the base of $%
D(\alpha,\theta)$. The geometry in the base is spherical. By construction of
the fundamental domain $D(\alpha,\theta)$, the base is a spherical triangle
with angles $2\pi /3, \alpha ,\alpha $, and the height $\zeta_{0}$ is $\left|
6\theta -\pi \right| $, where (\ref{ealfayteta})
\begin{equation*}
\alpha =\frac{\pi}{r(p+6q)}\qquad \theta= \frac{\pi (1-qr)}{pr}
\end{equation*}

Therefore
\begin{equation*}
\text{Area of }\Delta = 2\pi /3 +2 \alpha -\pi =\frac{2\pi}{r(p+6q)}-\pi /3=%
\frac{6\pi -\pi(r(p+6q))}{3r(p+6q)}
\end{equation*}

\begin{equation*}
V(\mathbb{T}_{p/q},r)=\left| \frac{6\pi (1-qr)}{pr}-\pi \right| \left( \frac{%
6\pi -\pi (r(p+6q))}{3r(p+6q)}\right) \frac{1}{4}=\left| \frac{\pi
^{2}(pr+6qr-6)^{2}}{12pr^{2}(p+6q)}\right| .
\end{equation*}

\begin{remark}
Observe that the volume $V(\mathbb{T}_{p/q},r)$, as could be suggested by the construction, is not the product of the area of the base by the height $\zeta_{0}$. There is a correcting factor of $(1/4)$ because the geometry is a twisted geometry.
\end{remark}

The following examples offer the volume of the conemanifolds represented by
some points in the plot $\mathcal{P}_{1}$.

\textbf{Example 1. $(-1,1)\in \mathcal{P}_{1}$, $(K=1,\,p=-1,\,q=1,\,r=1$).
Spherical geometry.}

\begin{equation*}
V(\mathbb{T}_{-1},1)=\frac{2\pi ^{2}}{120}
\end{equation*}

This manifold is the Poincar\'{e} spherical manifold, obtained by $(-1)$%
-surgery in the left-handed trefoil knot. This manifold is the quotient of
the sphere $\mathbb{S}^{3}$ by the binary icosahedral group $I\ast $, group
with $120$ elements. The volume of $\mathbb{S}^{3}$ is $2\pi ^{2}$.

\textbf{Example 2. $(5,0)\in \mathcal{P}_{1}$, $(K=1,\,p=1,\,q=0,\,r=5)$.
Spherical geometry.}

\begin{equation*}
V(\mathbb{T}_{\infty},5)=\frac{2\pi ^{2}}{600}=\frac{2\pi ^{2}}{120}\left/
5\right.
\end{equation*}

This result is coherent with the fact that the $5$-fold cyclic covering of $%
\mathbb{S}^{3}$ branched over the trefoil knot is the Poincar\'{e} spherical
manifold.

\textbf{Example 3. $(2,0)\in \mathcal{P}_{1}$, $(K=1,\,p=1,\,q=0,\,r=2)$.
Spherical geometry.}

\begin{equation*}
V(\mathbb{T}_{\infty},2)=\frac{2\pi ^{2}}{6}
\end{equation*}

This result is coherent with the fact that the double covering of $\mathbb{S}%
^{3}$ branched over the trefoil knot is the lens $L(3,1)$ which has $\mathbb{%
S}^{3}$ as its universal cover with $3$ sheets and therefore it has volume $%
\frac{2\pi ^{2}}{3}=2V(\mathbb{T}_{\infty },2)$.

\textbf{Example 4. $(3,0)\in \mathcal{P}_{1}$, $(K=1,\,p=1,\,q=0,\,r=3)$.
Spherical geometry.}

\begin{equation*}
V(\mathbb{T}_{\infty},3)=\frac{2\pi ^{2}}{24}
\end{equation*}
\begin{figure}[h]
\begin{center}
\epsfig{file=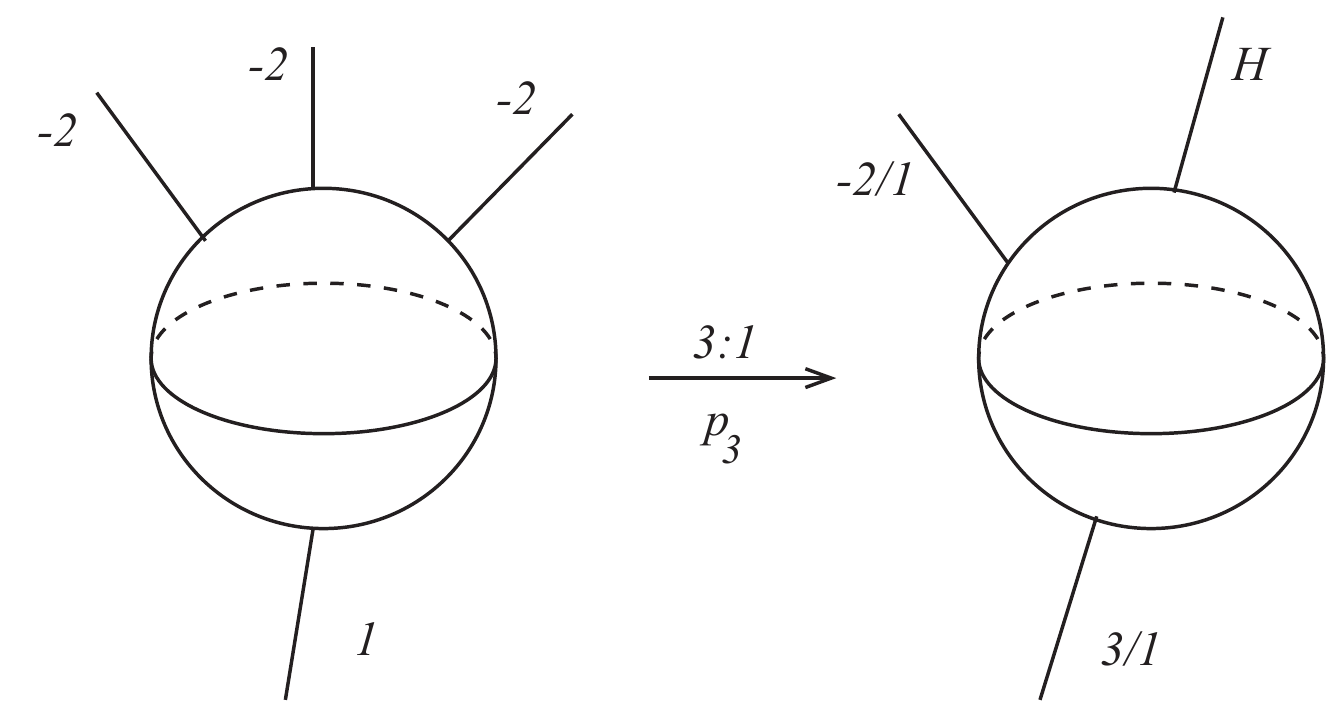,height=4cm}
\end{center}
\caption{The covering $p_{3}:(Oo0|1;-2,-2,-2)\longrightarrow
(Oo0|-1;2/1,3/1) $.}
\label{fkinf3}
\end{figure}

In Figure \ref{fkinf3} is depicted a scheme (compare \cite[Figure 12 p.146]{Monte1987}) of the
covering
\begin{equation*}
p_{3}:(Oo0|1;-2,-2,-2)\longrightarrow (Oo0|-1;2/1,3/1)
\end{equation*}%
branched over an ordinary fibre $H$ which is a trefoil knot. In Figure \ref%
{fkinf3} we have depicted the bases (both $\mathbb{S}^{2}$) of the Seifert
structures involved together with the exceptional fibres and an ordinary
fibre.
\begin{figure}[h]
\begin{center}
\epsfig{file=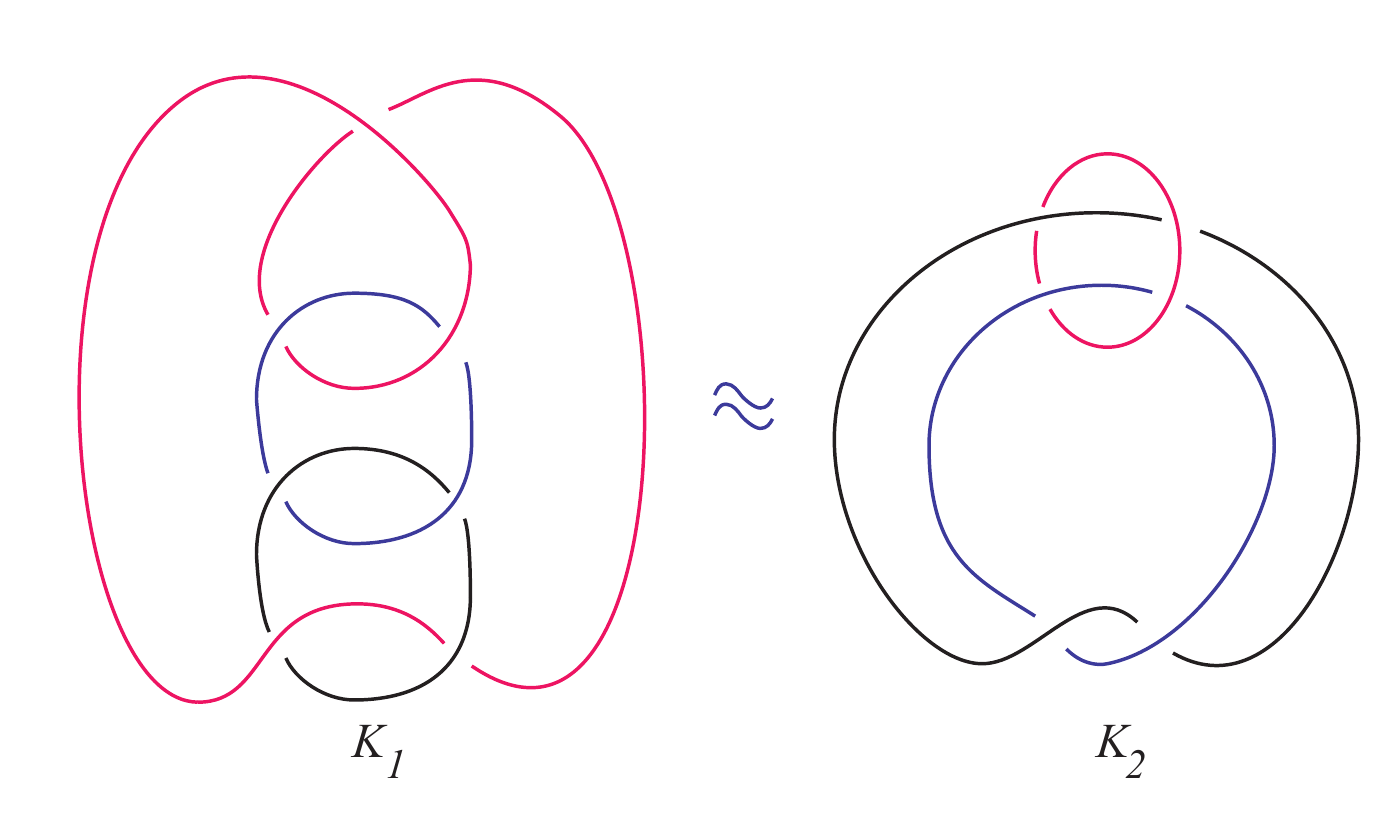,height=4cm}
\end{center}
\caption{The knots $K_{1}$ and $K_{2}$.}
\label{fk1k2}
\end{figure}

The map $p_{3}$ is the 3-fold cyclic covering of $\mathbb{S}^{3}$ branched
over the trefoil knot. On the other hand the Seifert manifolds $%
(Oo0|1;-2,-2,-2)$ and $(On1|2)$ are the 2-fold covering of $\mathbb{S}^{3}$
branched over the knots $K_{1}$ and $K_{2}$, depicted in Figure \ref{fk1k2}
respectively (\cite{Monte1973} and \cite{Monte1987}). Figure \ref{fk1k2} proves that the links $K_{1}$ and $K_{2}$ are
different diagrams of the same knot, therefore the Seifert manifolds are
equal:
\begin{equation*}
(Oo0|1;-2,-2,-2)=(On1|2).
\end{equation*}

This is the \emph{quaternionic space}, having $\mathbb{S}^{3}$ as universal
covering with $8$ sheets. Note that the group acting in $\mathbb{S}^{3}$
with quotient the quaternionic space $(On1|2)$ is the binary dihedral $%
<2\,2\,2>$ with 24 elements. Therefore $(\mathbb{T}_{\infty },3)$ is the
quotient of $\mathbb{R}P^{3}$ by the action of the tetrahedral group
(\cite{Monte1987}).

\textbf{Example 5. $(4,0)\in \mathcal{P}_{1}$, $(K=1,\,p=1,\,q=0,\,r=4)$.
Spherical geometry.}

\begin{equation*}
V(\mathbb{T}_{\infty},4)=\frac{2\pi ^{2}}{96}
\end{equation*}

The universal covering of $(\mathbb{T}_{\infty},4) $ factors through the octahedral manifold $\mathbb{S}^{3}/T^{\ast }$, where $T^{\ast }$ is the binary tetrahedral
group, and also, through the manifold $(Oo0|-4)=(Oo0|0;\overset{8}{%
\underbrace{1,...,1}},\overset{12}{\overbrace{-1,...,-1}})$.
\begin{eqnarray*}
\mathbb{S}^{3} &\overset{24:1}{\longrightarrow }&\mathbb{S}^{3}/T^{\ast
}=(Oo0|-2;3,3,3,3) \\
{}_{4:1}\downarrow & {}_{96:1}\searrow  &\downarrow {}_{4:1} \\
(Oo0|-4) &\overset{24:1}{\longrightarrow }&(\mathbb{T}_{\infty
},4)=(Oo0|0;-2,3)=\mathbb{S}^{3}
\end{eqnarray*}

\textbf{Example 6. $(\infty ,0)\in \mathcal{P}_{1}$, $(p=1,\,q=0,\,r=\infty)$%
. $\widetilde{SL(2,\mathbb{R})}$ geometry.}

\begin{equation*}
V(\mathbb{T}_{\infty},\infty)=\frac{\pi ^{2}}{12}
\end{equation*}
The geometry in the conemanifolds $(\mathbb{T}_{\infty},r)$, $r>6$ is $%
\widetilde{SL(2,\mathbb{R})}$. The volume of $(\mathbb{T}_{\infty},\infty)$
will be the limit of $V (\mathbb{T}_{\infty},r)$, $r\to \infty$:
\begin{equation*}
V(\mathbb{T}_{\infty},\infty)=\lim _{r\to \infty}V(\mathbb{T}%
_{\infty},r)=\lim _{r\to \infty}\frac{\pi^ {2}(-6+r)^{2}}{12r^{2}}=\frac{%
\pi^{2}}{12}
\end{equation*}

\textbf{Example 7. $(6,0)\in \mathcal{P}_{1}$, $(p=1,\,q=0,\,r=6)$. $X_{1}$
geometry}

$(\mathbb{T}_{\infty },6)$ is the Seifert manifold $%
(Oo0|-1;(2/1),(3/1),(6/1))$ according to (\ref{eseifertnil}). The singular
fibre is the core of the surgery (here the trefoil knot) with angle $\beta
=2\pi /6$. The $6$-fold covering of $\mathbb{S}^{3}$ branched over the
trefoil knot is the Seifert manifold $(Oo1|1)$ (\cite{Monte1973})
\begin{equation*}
V(\mathbb{T}_{\infty },6)=\lim_{r\rightarrow 6}V(\mathbb{T}_{\infty
},r)=\lim_{r\rightarrow 6}\frac{\pi ^{2}(-6+r)^{2}}{12r^{2}}=0
\end{equation*}%
.

\subsection{Plotting the geometric structures on the manifolds obtained by
surgery on the trefoil knot}

We are studying cone-manifold structures in manifolds obtained by
Dehn-surgery on the Trefoil knot in $\mathbb{S}^{3}$. These manifolds are
Seifert manifolds with geometric structure and with the core of the surgery
as singular set. It seems natural to adopt the following new notation.

\begin{notation}
Let
\begin{equation*}
S(m,n)=\left( Oo0|-1;(2,1),(3,1),(m,n)\right) \qquad m\geq 0
\end{equation*}%
be the Seifert manifold $\left( Oo0|-1;(2,1),(3,1),(\frac{m}{r},\frac{n}{r}%
)\right) $, where $r=\gcd (m,n)$, with conic singularity along the
exceptional fibre $(\frac{m}{r},\frac{n}{r})$ of angle $2\pi /r$.
\end{notation}

Recall that the holonomy of $S(m,n)$ is generated by $a(\alpha ,\theta )$ (%
\ref{eaAij}) and $b(\alpha ,\theta )$ (\ref{ebBij}), where
\begin{equation*}
\left\{
\begin{array}{c}
\frac{2\pi }{r}=2\alpha \frac{m}{r} \\
\\
\frac{6\theta -\pi }{\theta -\alpha }=\frac{m}{n}%
\end{array}%
\right.
\end{equation*}%
if $\alpha \neq \pi /6$. The geometry is $\widetilde{SL(2,\mathbb{R})}$ for $%
0\leq \alpha <\frac{\pi }{6}$; Nil  for $\alpha =\frac{\pi }{6}$; and spherical for $\frac{\pi }{6}<\alpha <%
\frac{5\pi }{6}$.

Because
\begin{equation*}
\frac{m}{n}=6+\frac{p}{q}=\frac{6q+p}{p}\quad \Longrightarrow \quad 8\frac{q%
}{p}+1=(\frac{m+2n}{m-6n}),
\end{equation*}%
where $p/q$ is the datum of the surgery, the holonomy (\ref{eaqpbqp}) for $%
\alpha =\frac{\pi }{6}$ is given by the matrices
\begin{eqnarray*}
a_{m/n} &=&\left[
\begin{array}{cccc}
\frac{1}{2} & -\frac{\sqrt{3}}{2} & 0 & \frac{1}{2} \\
\frac{\sqrt{3}}{2} & \frac{1}{2} & 0 & -\frac{\sqrt{3}}{2} \\
\frac{\sqrt{3}}{2} & -\frac{1}{2} & 1 & -\frac{1}{2}\sqrt{3}(\frac{m+2n}{m-6n%
}) \\
0 & 0 & 0 & 1%
\end{array}%
\right]  \\
&& \\
b_{m/n} &=&\left[
\begin{array}{cccc}
\frac{1}{2} & -\frac{\sqrt{3}}{2} & 0 & -\frac{1}{2} \\
\frac{\sqrt{3}}{2} & \frac{1}{2} & 0 & \frac{\sqrt{3}}{2} \\
-\frac{\sqrt{3}}{2} & \frac{1}{2} & 1 & -\frac{1}{2}\sqrt{3}(\frac{m+2n}{m-6n%
}) \\
0 & 0 & 0 & 1%
\end{array}%
\right]
\end{eqnarray*}

With the notation $S(m,n)$
\begin{eqnarray*}
\left( \mathbb{T}_{p/q},r\right) &=& S(r(6q+p),rq) \\
S(m,n) &=& \left( \mathbb{T}_{(m-6n)/n},m.c.d.(m,n)\right) \\
\alpha &=& \frac{\pi}{m} \\
\theta &=& \pi \frac{n-1}{6n-m} \\
V(S(m,n)) &=& -\frac{\pi ^{2}}{S}\frac{(m-6)^{2}}{12m(m-6n)} \\
V(S(\infty m,\infty n)) &=& -\frac{\pi ^{2}}{S}\frac{m}{12(m-6n)}
\end{eqnarray*}

\begin{equation*}
\left\{
\begin{array}{c}
6<m\leq \infty  \\
m=6 \\
\frac{6}{5}<m<6%
\end{array}%
\right\} \quad \text{geometry}\quad \left\{
\begin{array}{c}
\widetilde{{SL(2,\mathbb{R})}} \\
\text{Nil} \\
\text{spherical}%
\end{array}\right\}
\end{equation*}%
One of the advantages of using the notation $S(m,n)$ is that the information
about the different geometries on the same Seifert manifold can best be seen
in the following plot $\mathcal{P}_{2}$ (Figure \ref{fplot2}), written in
the spirit of Thurston (see \cite[p. 4.21]{Thurston1}). The points in the plot bearing the symbol $\blacklozenge $ correspond to
points $(x,y)$ such that $x$ and $y$ are non-negative integers with $\gcd
(x,y)=1$. They represent the manifolds $S(x,y)$ with no singularity. The
points in the line $x/y$ connecting the origin $(0,0)$ with the point $(x,y)$
represent $S(x\times r,y\times r)$. This is a cone-manifold with underlying $%
3$-manifold the common Seifert manifold $S(x,y)$, but the cone angle $2\pi /r
$ varies from zero to infinity. The left
vertical dashed line $ \mathcal{U}$ corresponds to the set of upper limits $ \mathcal{U}$ in the plot $ \mathcal{P}_{1}$ and it separates the left region, with unknown, if any, geometry compatible with the  Seifert structure in the sense that fibres must be geodesics,
from the region whose points represent spherical geometry. Points in the right
vertical dashed line $\mathcal{L}$, which correspond to the lower limits $\mathcal{L}$ in plot $\mathcal{P}_{1}$, represent  Nil geometry and it separates the region whose points represent
spherical geometry from  the region whose points represent $\widetilde{{SL(2,\mathbb{R})}}$ geometry.

\begin{figure}[h]
\begin{center}
\epsfig{file=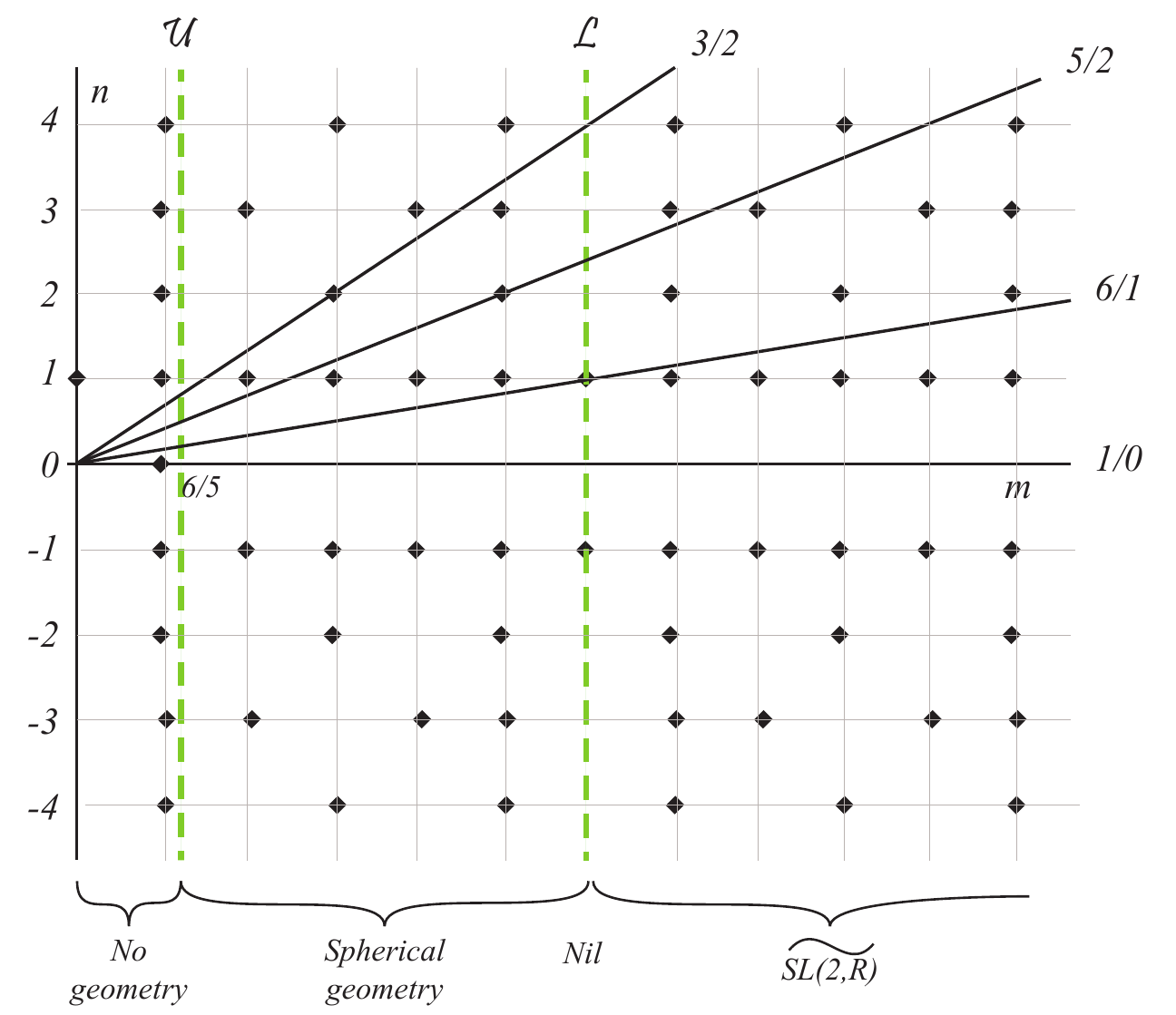,height=8cm}
\end{center}
\caption{The plot $\mathcal{P}_{2}$.}
\label{fplot2}
\end{figure}

\begin{definition}
The \emph{limit of sphericity} $\gamma _{e}$ of the Seifert manifold $S(x,y)$ is the cone angle of the cone-manifold structure corresponding to the intersection point  $ \mathcal{U}\cap (x/y)$.

The \emph{Nil angle} $\gamma _{N}$ of the Seifert manifold $S(x,y)$ is the cone angle of the cone-manifold structure corresponding to the intersection point $ \mathcal{L}\cap (x/y)$. Their value as a function of $(x,y)$ is the following
\[ \gamma _{e}=\frac{5x\pi}{3} \quad \quad \gamma _{N}=\frac{x\pi }{3}
\]
\end{definition}

Next we summarize some affirmations that are deduced easily from  the plot $\mathcal{P}_{2}$.
 \begin{summary}
 Let $x$ and $y$ be non-negative integers with $\gcd
(x,y)=1$.
\begin{itemize}
  \item The geometric cone-manifold structure with cone angle $\gamma$ in $S(x,y)$ is spherical for $\gamma _{e}>\gamma >\gamma _{N}$, Nil for $\gamma =\gamma _{N}$, and $\widetilde{{SL(2,\mathbb{R})}}$ for $\gamma _{N}>\gamma $.
  \item The limit of sphericity $\gamma _{e}$ is equal to five times the Nil angle $\gamma _{N}$
  \item There exits infinity many non-singular Nil geometric structures: \newline $S(6,y)=\left( \mathbb{T}_{(6-6y)/y},1\right)$, where $\gcd (6,y)=1$.
  \item There exits infinity many Nil orbifold geometric structures: \begin{enumerate}
                                                                       \item $S(6,3y)=\left( \mathbb{T}_{(2-6y)/y},3\right)$, where $\gcd (2,y)=1$,
                                                                       \item  $S(6,2y)=\left( \mathbb{T}_{(3-6y)/y},2\right)$, where $\gcd (3,y)=1$.
                                                                     \end{enumerate}
  \item If $\left( \mathbb{T}_{p/q},r\right) = S(r(6q+p),rq)$ has a spherical orbifold structure, then $2\leq r \leq 5$.
\end{itemize}
 \end{summary}
Lets analyze some points in the region $m\in [0, 6/5]$. The vertical axis $x=0$ yields $S(0,y)$, corresponding to the manifold $%
L(2,1)\sharp L(3,1)$ with no geometry. Between this axis and the vertical
line $m=6/5$ are placed the Seifert manifolds (no singular)
\begin{equation*}
S(1,n)=(Oo0|-1;(2,1),(3,1),(1,n))=(Oo0|n-1;(2,1),(3,1))
\end{equation*}
which are the lens $-L(6n-1,2n-1)$ for $n\neq 0$, and $\mathbb{S}^{3}$ for $%
n=0$. It is known that these Seifert manifolds (lens spaces) support
spherical geometry but the fibre $(1,n)$ of the Seifert fibration $%
(Oo0|-1;(2,1),(3,1),(1,n))$ is not a geodesic of that geometry. For instance, in $S(1,0)=\mathbb{S}^{3}$,  the fibre $(1,0)$ is a regular
fibre, the left Trefoil knot, which obviously
is not a geodesic in $\mathbb{S}^{3}$. Recall that we are
studying geometric structures in the Seifert manifold $S(x,y)$ such that the
fibre $(x,y)$ is a geodesic (singular or not), that is Seifert geometric conemanifolds.

\end{document}